\documentclass{article}

 \usepackage{graphicx}

\usepackage{amsmath, amssymb, amsthm, enumerate}
\usepackage{multirow}
\usepackage{amsmath,amssymb,amsfonts}
\usepackage{mathrsfs}
\usepackage[title]{appendix}
\usepackage{xcolor}
\usepackage{textcomp}
\usepackage{manyfoot}
\usepackage{booktabs}
\usepackage{algorithm}
\usepackage{algorithmicx}
\usepackage{algpseudocode}
\usepackage{listings}
\usepackage{bm}
\usepackage{hyperref}
\usepackage{url}
\usepackage{amsmath,amssymb,amsthm,enumerate,mathrsfs}
\newtheorem{theorem}{Theorem}[section]

\newtheorem{proposition}[theorem]{Proposition}
\theoremstyle{definition}

\newtheorem{example}[theorem]{Example}

\newtheorem{remark}[theorem]{Remark}

\numberwithin{equation}{section}

\begin{document}
\makeatletter

\begin{center}
\large{\bf Fast  projection onto the intersection of  simplex and singly linear  constraint and its generalized Jacobian}
\end{center}\vspace{5mm}
\begin{center}

\textsc{Weimi Zhou\footnote{ School of Mathematics and Statistics, Fuzhou University, No. 2 Wulongjiang North Avenue, Fuzhou, 350108, Fujian, China. Email: wmzhou1997@163.com}, Yong-Jin Liu\footnote{Corresponding author. Center for Applied Mathematics of Fujian Province, School of Mathematics and Statistics, Fuzhou University, No. 2 Wulongjiang North Avenue, Fuzhou, 350108, Fujian, China. Email: yjliu@fzu.edu.cn}}\end{center}

\vspace{2mm}

\footnotesize{
\noindent\begin{minipage}{14cm}
{\bf Abstract:}
Solving the distributional worst-case in the distributionally robust optimization problem is equivalent to finding the projection onto the intersection of simplex and singly linear inequality constraint. This projection is a key component in the design of  efficient first-order algorithms. This paper focuses on developing efficient algorithms for computing the projection onto the intersection of simplex and singly linear inequality constraint.  Based on the Lagrangian duality theory, the studied projection can be obtained by solving a univariate nonsmooth equation. We employ an algorithm called LRSA, which leverages the Lagrangian duality approach and the secant method to compute this projection.  In this algorithm, a modified secant method is specifically designed to solve the piecewise linear equation.  Additionally, due to semismoothness of the resulting equation, the semismooth Newton (SSN) method is a natural choice  for solving it. Numerical experiments demonstrate that LRSA outperforms SSN algorithm and the state-of-the-art optimization solver called Gurobi. Moreover, we derive  explicit formulas for the generalized HS-Jacobian of the  projection, which are essential for designing  second-order nonsmooth Newton algorithms.
\end{minipage}
 \\[5mm]

\noindent{\bf Keywords:} {Projection, Simplex,  Secant method, Semismooth Newton method, Generalized Jacobian}\\
\noindent{\bf Mathematics Subject Classification:} {90C20, 90C25}

\hbox to14cm{\hrulefill}\par


\section{Introduction}
Let $\Delta_{n-1}$  be the simplex in $\mathbb{R}^n$ given by $\Delta_{n-1}: =\{\bm{x}\in\mathbb{R}^n\ |\ \bm{e}^{\top}\bm{x}= 1,\bm{x}\geq 0 \},$ where $\bm{e}$ denotes  the column vector of all ones. Given $\bm{y}\in\mathbb{R}^n, \bm{a}\in\mathbb{R}^n,$ and  $b\in\mathbb{R}$, we are concerned with the  efficient  projection of the vector $\bm{y}$ onto the intersection of simplex and singly linear inequality constraint, i.e., 
$$
\mathcal{C} = \{\bm{x}\in\mathbb{R}^n \ |  \ \bm{a}^{\top}\bm{x}\leq b , \bm{x}\in\Delta_{n-1}\}.
$$
It is well known that the projection onto the set ${\mathcal{C}}$, denoted by $\Pi_{\mathcal{C}}(\bm{y})$, is the unique optimal solution to the following optimization problem:
\begin{equation}\label{1.1}\tag{P}
	\mathop{\min}\limits_{\bm{x}\in\mathcal{C}} \ \frac{1}{2}\|\bm{x}-\bm{y}\|^2.	
\end{equation}
The motivation for studying problem \eqref{1.1} comes from the characteristics of the model and the design of efficient algorithms in the field of distributionally robust optimization. The purpose of distributionally robust optimization is to find a decision that minimizes the expected cost in the worst case. In terms of the characteristics of the distributionally robust optimization model,  Adam and  M{\'a}cha \cite{L.V.2022} demonstrated that solving  the distributional worst-case  is equivalent to computing the projection onto the set $\mathcal{C}$.  In addition, the set $\mathcal{C}$ involved in the constraints often appears in distributionally robust optimization portfolio models.  Chen et al. \cite{C.W.L.D.C.2022} transformed  the  Wasserstein metric-based data-driven distributionally robust mean-absolute deviation model into two simple finite-dimensional linear programs, one of which is expressed as follows:
\begin{equation}\label{1.2}
	\begin{split}
		\mathop{\min}\limits_{\bm{x}\in\mathbb{R}^n}& \ \max\{\frac{1}{m}\sum_{i=1}^{m}|\tilde{\bm{\mu}}^{\top}\bm{x}-(\hat{\bm{\xi}}_i)^{\top}\bm{x}-\epsilon|+\epsilon, \frac{1}{m}\sum_{i=1}^{m}|\tilde{\bm{\mu}}^{\top}\bm{x}-(\hat{\bm{\xi}}_i)^{\top}\bm{x}+\epsilon|+\epsilon\}   \\
		\mbox{s.t.}\  & \ \tilde{\bm{\mu}}^{\top}\bm{x}\geq \tilde{\rho}+\epsilon,\ x\in\Delta_{n-1},
	\end{split}
\end{equation}
where  the random vectors $\hat{\bm{\xi}}_1, \dots, \hat{\bm{\xi}}_m\in\mathbb{R}^n$ denote the assets return, $\tilde{\bm{\mu}}=1/m\sum_{i=1}^{m}\hat{\bm{\xi}}_i$, $\epsilon$ represents the radius of the Wasserstein ball,  and $ \tilde{\rho}$ is given. From the perspective of algorithmic design, it is necessary to compute the projection onto the set $\mathcal{C}$ when  applying methods such as the augmented Lagrangian method or the proximal gradient method to solve problems involving a simplex and a single linear constraint.  Furthermore, the corresponding generalized Jacobian of the projection onto a closed convex set  is a necessary ingredient in some second-order nonsmooth algorithms \cite{LXD.S.T.2018,L.S.T.2020,MXL.YJL.DFS.KCT.2019,FS.LYJ.XXZ.2021,LMX.SDF.TKC.2019,LMX.SDF.TKC.2022}.

This paper aims to design  efficient algorithms for finding the projection onto the intersection of  simplex and  singly linear inequality constraint. Moreover, we intend to derive the explicit form of its generalized Jacobian matrix. Our main idea is inspired by the seminal work \cite{L.L.2017,L.S.T.2018,W.L.L.2022,L.X.L.2022,L.W.S.S.2013,C.L.S.T.2016,LXD.S.T.2018}. For  projections onto the intersection of linear constraint and box-like constraint, the algorithms demonstrate excellent performance.   Combining the sparse reconstruction by separable approximation (SpaRSA)  \cite{SJW.RDN.MATF.2009} with the dual active set algorithm, Hager and Zhang \cite{WWH.HZ.2016} proposed an algorithm for projecting a given point onto  polyhedron. They provided an efficient implementation called PPROJ, which is used in \cite{WWH.HZ.2023,DD.HWW.TG.VM.2023}. The modified secant algorithm proposed by Dai and Fletcher \cite{D.F.2006} is applied to calculate the projection onto the intersection of  singly equality constraint and  box constraint. Liu and Liu \cite{L.L.2017} proposed an efficient algorithm on the basis of parameter approach and  modified secant method for solving singly linearly constrained quadratic programs with box-like constraint,   which outperforms advanced solvers such as Gurobi and Mosek.  Wang et al. \cite{W.L.L.2022} developed a fast algorithm based on Lagrangian dual method and semismooth Newton method, where  the semismooth Newton method replaces the modified secant method in \cite{L.L.2017} for  finding the root of a piecewise affine equation. The  algorithm based on Lagrangian duality approach and secant method  has good numerical performance  in calculating the projection onto the polyhedron. Moreover, the semismooth Newton method is  a common second-order method for computing the projection. Therefore, we would like  to  adopt these two algorithms  to compute the projection onto the intersection of simplex and singly linear inequality constraint.

The main contributions of our paper are summarized as follows. Firstly, we provide theoretical results on the projection onto a simplex and adopt an efficient algorithm for computing  it.  Secondly,  leveraging the Lagrangian duality theory, the optimal solution to problem \eqref{1.1} can be obtained by solving a univariate nonsmooth equation. We develop an algorithm based on the Lagrangian duality approach and modified secant method (referred to as LRSA), where the modified secant method is specifically designed to solve the piecewise linear equation. Additionally, we design an algorithm based on the Lagrangian duality  and semismooth Newton method (referred to as SSN) to compute the studied projection, in which the semismooth Newton method is applied to solve the nonsmooth equation. We also derive the generalized Clarke's differential required for the semismooth Newton method. Thirdly,
 we conduct experiments on  problem \eqref{1.1}  and demonstrate the superiority of LRSA  by comparing algorithms LRSA and SSN with the state-of-the-art solver Gurobi \cite{Gurobi}. Finally, we derive the generalized HS-Jacobian \cite{H.S.1997} of the projection we studied, which is essential for the future development of the semismooth Newton proximal point algorithm to solve the related constrained problems.

The remaining parts of this paper are organized as follows. 
Section \ref{sect:2} is devoted to presenting some properties of the metric projection onto the simplex. In Section \ref{sect:3}, we  derive the  dual of the projection problem \eqref{1.1} and analyze some properties of its objective function. Building on these foundations, we develop two algorithms for solving the dual problem: one is based on the secant method, and the other is based on the semismooth Newton method.
 In Section \ref{sect:5}, we compare these two algorithms with Gurobi on  random data and real data.  Section \ref{sect:6} is dedicated to computing the generalized HS-Jacobian of $\Pi_{\mathcal{C}}(\cdot)$.  We make the conclusion of this paper in Section \ref{sect:7}.

{\bf{Notation:}}\ 
For given positive integer $m$, we denote $\bm{I}_m$ and ${\bf{0}}_m$ as the $m\times m$ identity matrix and the $m\times m$ zeros matrix, respectively. For given $\bm{x}\in \mathbb{R}^n$, we use $``{\rm{sgn}}(\bm{x})"$ to denote the sgn vector whose $i$-th entry is $1$ if  $\bm{x}_i>0,$ $-1$ if $\bm{x}_i<0$, and $0$ otherwise. Denote ${\rm{Diag}}(\bm{x})$ as the diagonal matrix whose diagonal is given by vector $\bm{x}$. Given a matrix $\bm{B}\in\mathbb{R}^{n \times m}$, we  denote  the Moore-Penrose inverse of $\bm{B}$ by $\bm{B}^{\dagger}$. For given index set ${\mathcal{I}}\subseteq\{1,2,\dots,n\}$, we use $|{\mathcal{I}}|$ to define the cardinality of $\mathcal{I}$, and  use $\bm{B}_{\mathcal{I}}$ to denote the sub-matrix of $\bm{B}$ by extracting all the rows of $\bm{B}$ in $\mathcal{I}$. $\max(\bm{a})$ denotes the maximum component of the column vector  $\bm{a}$.

\section{The projection onto the simplex}\label{sect:2}
In this section, we review  key results on the projection onto the simplex and apply an efficient algorithm to compute this projection, which is essential for designing efficient algorithms for problem \eqref{1.1}.

Given $\bm{y}\in\mathbb{R}^n$, the projection of $\bm{y}$ onto the set ${\Delta_{n-1}}$, denoted by $\Pi_{\Delta_{n-1}}(\bm{y})$, is given by 
\begin{equation}\label{1.3}
	\Pi_{\Delta_{n-1}}(\bm{y}):=\mathop{\arg\min}\limits_{\bm{x}\in\Delta_{n-1}} \ \frac{1}{2}\|\bm{x}-\bm{y}\|^2.
\end{equation}
Suppose that  $\chi_{\Delta_{n-1}}:\mathbb{R}^n\to(-\infty,+\infty]$ is the indicator function of  the set $\Delta_{n-1}$. Then the Moreau envelope of $\chi_{\Delta_{n-1}}$ is defined by 
\begin{equation}\label{MP}
	M_{\chi_{\Delta_{n-1}}}(\bm{y}):=\mathop{\min}\limits_{\bm{x}\in\mathbb{R}^n} \ \left\{ \frac{1}{2}\|\bm{x}-\bm{y}\|^2+\chi_{\Delta_{n-1}}(\bm{x})\right\}=\frac{1}{2}\|\Pi_{\Delta_{n-1}}(\bm{y})-\bm{y}\|^2,\ \forall \bm{y}\in\mathbb{R}^n.
\end{equation}
The properties   associated with $\Pi_{\Delta_{n-1}}(\cdot)$ and $M_{\chi_{\Delta_{n-1}}}(\cdot)$ are stated in the next proposition (cf.   \cite{Z.E.H.1971}).
\begin{proposition}\label{pro1}
	The following properties hold:
	\begin{itemize}
		\item [(1)] The  projection $\Pi_{\Delta_{n-1}}(\cdot)$ satisfies
		$$
		\|\Pi_{\Delta_{n-1}}(\bm{x})-\Pi_{\Delta_{n-1}}(\bm{y})\|^2\leq \left<\Pi_{\Delta_{n-1}}(\bm{x})-\Pi_{\Delta_{n-1}}(\bm{y}),\bm{x}-\bm{y}\right>, \ \forall \bm{x},\bm{y}\in\mathbb{R}^n.
		$$
		Hence, the  projection $\Pi_{\Delta_{n-1}}(\cdot)$ is globally Lipschitz continuous with modulus $1$.
		\item [(2)]  The Moreau envelope $M_{\chi_{\Delta_{n-1}}}(\cdot)$ is convex, continuously differentiable, and its gradient at $\bm{y}$ is 
		$$
		\nabla M_{\chi_{\Delta_{n-1}}}(\bm{y})=\bm{y}-\Pi_{\Delta_{n-1}}(\bm{y}).
		$$
		Moreover, $\nabla M_{\chi_{\Delta_{n-1}}}(\cdot)$ is globally Lipschitz continuous with modulus $1$.
	\end{itemize}
\end{proposition}
As studied in \cite{C.L.2016}, we obtain the following important results of $\Pi_{\Delta_{n-1}}(\cdot)$.
\begin{proposition}\label{pro2}
	Let $\bm{y}\in\mathbb{R}^n$ be a given vector. Then there exists a unique $\tau\in\mathbb{R}$ such that
	$$
	\Pi_{\Delta_{n-1}}(\bm{y})=\max(\bm{y}-\tau,0).
	$$
\end{proposition}
Numerous algorithms \cite{H.W.C.1974,V.F.2008,K.K.C.2008,M.C.1986}  have been proposed for computing the projection onto the simplex. In our numerical experiment, we choose the algorithm proposed by  Condat \cite{C.L.2016} to calculate the projection $\Pi_{\Delta_{n-1}}(\cdot)$. 
\section{Efficient algorithms based on Lagrangian duality  method}\label{sect:3}
In this section, we  present the Lagrangian dual  of problem \eqref{1.1} and design two efficient algorithms based on the Lagrangian dual theory to find an optimal solution of problem \eqref{1.1}.

\subsection{Lagrangian duality method}
Recall that problem \eqref{1.1} can be rewritten as:
\begin{equation}\label{3.1}
	\begin{split}
		\mathop{\min}\limits_{\bm{x}\in\mathbb{R}^n}&  \ \frac{1}{2}\|\bm{x}-\bm{y}\|^2\\
		\mbox{s.t.}\ &\  \bm{a}^{\top}\bm{x}\leq b,\\
		\quad & \ \bm{x}\in\Delta_{n-1}.
	\end{split}
\end{equation}
The corresponding Lagrangian function of problem \eqref{3.1} in the extended form is given by 
$$
L(\bm{x};\sigma):=
\begin{cases}
	\frac{1}{2}\|\bm{x}-\bm{y}\|^2+\sigma(\bm{a}^{\top}\bm{x}-b), & \bm{x}\in\Delta_{n-1},\\ 
	+\infty, & {\rm{otherwise}}.
\end{cases}
$$
The dual of problem \eqref{3.1}  is formulated as follows:
\begin{equation}\label{dual}
	\max\limits_{\sigma\geq 0}\ h(\sigma),
\end{equation}
where the objective function $h(\cdot)$ is defined by 
\begin{equation}\label{dualfun}
		h(\sigma):=\min\limits_{\bm{x}\in\Delta_{n-1}}\{L({\bm{x}};\sigma)\}
		=\min\limits_{\bm{x}\in\Delta_{n-1}}\{\frac{1}{2}\|\bm{x}-(\bm{y}-\sigma \bm{a})\|^2\}-\frac{1}{2}\|\bm{y}-\sigma \bm{a}\|^2+\frac{1}{2}\|\bm{y}\|^2-\sigma b.
\end{equation}
This implies that
$$
	h(\sigma)=M_{{\chi_{\Delta_{n-1}}}}(\bm{y}-\sigma \bm{a})-\frac{1}{2}\|\bm{y}-\sigma \bm{a}\|^2+\frac{1}{2}\|\bm{y}\|^2-\sigma b,\ \forall \sigma\geq 0,
$$
where $M_{{\chi_{\Delta_{n-1}}}}(\cdot) $ is defined  in \eqref{MP}.  It is easy to see that $\Pi_{\Delta_{n-1}}(\bm{y}-\sigma \bm{a})$ is the unique optimal solution to problem \eqref{dualfun}. Assume that $\bm{a}\neq\bm{0}$.  For a fixed $\sigma$, let $\bar{\bm{x}}(\sigma)$ denote  the unique optimal solution to problem \eqref{dualfun}, i.e., 
$	\bar{\bm{x}}(\sigma) = \Pi_{\Delta_{n-1}}(\bm{y}-\sigma \bm{a})$.  
Given $\sigma\geq0$, let $\kappa^{\sigma}$ (depending on $\sigma$) be a permutation of $\{1,\dots,n\} $ such that 
$$
\left[ {\bm{y}}-\sigma {\bm{a}}\right]_{\kappa^{\sigma}(1)}\geq\left[ {\bm{y}}-\sigma {\bm{a}}\right]_{\kappa^{\sigma}(2)}\geq\cdots\geq\left[ {\bm{y}}-\sigma {\bm{a}}\right]_{\kappa^{\sigma}(n)},
$$
where $\left[ {\bm{y}}-\sigma {\bm{a}}\right]_{\kappa^{\sigma}(i)}:=\left[ {\bm{y}}_{\kappa^{\sigma}(i)}-\sigma {\bm{a}}_{\kappa^{\sigma}(i)}\right],\forall i\in\{1,\dots,n\}.$
Now, we obtain the closed-form expression of $\Pi_{\Delta_{n-1}}(\bm{y}-\sigma \bm{a})$ from \cite{H.W.C.1974}. Define 
$$
\bar{K}(\sigma):=\max\left\lbrace 1 \leq j\leq n\mid ({\bm{y}}-\sigma {\bm{a}})_{\kappa^{\sigma}(j)}+\frac{1}{j}\left(  1-\sum_{i=1}^{j}({\bm{y}}-\sigma {\bm{a}})_{\kappa^{\sigma}(i)}\right)>0 \right\rbrace.
$$
 Then, for $i\in\{1,\dots,n\}$, we have 
\begin{equation}\label{xstart}
	\begin{aligned}
			\bar{\bm{x}}_i(\sigma) &= \left[ \Pi_{\Delta_{n-1}}(\bm{y}-\sigma \bm{a})\right]_i\\
			&=\left\{\begin{aligned}
					(\bm{y}-\sigma \bm{a})_i-\frac{\sum_{j=1}^{ \bar{K}(\sigma)}(\bm{y}-\sigma \bm{a})_{\kappa^{\sigma}(j)}-1}{\bar{K}(\sigma)},\  & (\bm{y}-\sigma \bm{a})_i-\frac{\sum_{j=1}^{ \bar{K}(\sigma)}(\bm{y}-\sigma \bm{a})_{\kappa^{\sigma}(j)}-1}{\bar{K}(\sigma)}>0,\\
					0,\ &{\rm{otherwise}}.\end{aligned}\right.
		\end{aligned}
\end{equation}

Now, we state some properties of the objective function $h(\cdot)$,  which provide theoretical basis for designing  efficient algorithms.
\begin{proposition}\label{pro7}
	Assume that $\bm{y},\bm{a}\in\mathbb{R}^n$ and $b\in\mathbb{R}$ are given. Then, the following properties are valid:
	\begin{itemize}
		\item [(1)] The function $h(\cdot)$ of the dual problem \eqref{dual} is coercive, closed, and  concave. Furthermore, the function $h(\cdot)$ is continuously differentiable with its gradient given by
		$$
		h'(\sigma)=\bm{a}^{\top}\Pi_{\Delta_{n-1}}(\bm{y}-\sigma \bm{a})-b.
		$$
		\item [(2)] For a given non-negative number $\sigma$, if $\sigma$ satisfies one of the following conditions:  (i) $\sigma=0$ and $h'(0)\leq0$, or (ii) $h'(\sigma)=0$, then the optimal solution of problem \eqref{dualfun} is the unique optimal solution of problem \eqref{1.1}.
	\end{itemize}
\end{proposition}
\begin{proof}
	The proof of item (1) and  (2) follow from \cite{RR.T.1974,GLM.1950} and  \cite{L.L.2017}, respectively. Here, we omit the details.
\end{proof}

For convenience, we define the function $\psi:\mathbb{R}_{+}\to\mathbb{R}$ by 
\begin{equation}\label{phi}
	\psi(\sigma):=h'(\sigma)=\bm{a}^{\top}\Pi_{\Delta_{n-1}}(\bm{y}-\sigma \bm{a})-b.
\end{equation}
From Proposition \ref{pro7}, we identify that the key to solving problem \eqref{1.1} is to determine $\sigma^*\in\mathbb{R}_{+}$ that satisfies either (i) $\sigma=0$ and $\psi(0)\leq0$, or (ii) $\psi(\sigma)=0$. These conditions are equivalent to $\sigma\psi(\sigma)=0$ with $\sigma\geq0.$ In fact, $\sigma^*\psi(\sigma^*)=0,\sigma^*\geq0$ and $x^* = \Pi_{\Delta_{n-1}}(y-\sigma^*a)$  are the KKT conditions at the optimal primal-dual pairs $(x^*,\sigma^*)$. Moreover,  we  present a useful property of $\psi$ for the sake of subsequent analysis.

\begin{proposition}\label{pro4}
Let the function $\psi$ be defined by \eqref{phi}. Then, $\psi$ is a continuous and monotonically nonincreasing function on $\mathbb{R}_{+}$.
\end{proposition}
\begin{proof}
Since $\Pi_{\Delta_{n-1}}(\cdot)$ is a globally Lipschitz continuous function, it follows that $\psi$ is   continuous.   The rest of the proof follows from \cite{L.L.2017}, thus we omit the details here.
\end{proof}

In view of the above analysis, our primary goal is to design efficient algorithms for solving the equation $\psi(\sigma)=0$ to obtain an optimal solution of the projection problem \eqref{1.1}. For the monotonically nonincreasing univariate function $\psi$, we can adopt secant method to find the  solution of $\psi(\sigma)=0$. On the other hand, since $\psi$ is not differentiable, the classical Newton method is not suitable for solving $\psi(\sigma)=0$. Considering that $\psi$ is strongly semismooth,  we  also attempt to  solve the nonsmooth equation $\psi(\sigma)=0$  using the semismooth Newton method \cite{Q.S.1993}. The details of these two algorithms will be presented in the next two subsections.

	\subsection{An algorithm based on Lagrangian duality approach and secant method}
	In this subsection, we apply an algorithm based on the Lagrangian duality approach and secant method (LRSA) \cite{L.L.2017} to find the projection $\Pi_{\mathcal{C}}(\cdot)$.  In Algorithm \ref{al:1}, given $\bm{y}\in\mathbb{R}^n$, we compute $\psi(0)=\bm{a}^{\top}\Pi_{\Delta_{n-1}}(\bm{y})-b$ and evaluate its value.
	If $\psi(0)\leq0$,  then $\Pi_{\mathcal{C}}(\bm{y}) = \Pi_{\Delta_{n-1}}(\bm{y});$ otherwise, we need to solve the equation $\psi(\sigma)=0$. The procedure of solving the equation $\psi(\sigma)=0$ is divided into two steps. In the first step, the bracketing phase method is used to find an interval $[\sigma_l,\sigma_u]$ that contains a root of the equation $\psi(\sigma)=0$. 
	In the second step, the secant method is applied to find the approximate solution  $\hat{\sigma}$ of the equation $\psi(\sigma)=0$ within the interval $[\sigma_l,\sigma_u]$. Finally, we obtain the projection $\Pi_{\mathcal{C}}(\bm{y}) = \Pi_{\Delta_{n-1}}(\bm{y}-\hat{\sigma}\bm{a}).$
	
	\begin{algorithm}[h]
		\caption{An algorithm based on Lagrangian duality approach and secant method (LRSA)  for  $\Pi_{\mathcal{C}}(\cdot)$}\label{al:1}	
		\begin{algorithmic}[1]
			\Require Given the vector $\bm{y}\in\mathbb{R}^n$, the parameters $\rho>1,\Delta\sigma>0$, and tolerance $\epsilon>0$.
			\State {\bf{Initialization Begins:}} Compute $\Pi_{\Delta_{n-1}}(\bm{y})$ and $r=\psi(0)$.
			\If{$r\leq0$}
			\State stop, and $\Pi_{\mathcal{C}}(\bm{y}) = \Pi_{\Delta_{n-1}}(\bm{y}).$ 
			\Else
			\State set $\sigma_l=0,r_l = r$.
			\EndIf
			\State{\bf{Bracketing Phase Begins:}}
			\For{$j = 0,1,\ldots$}
			\State Set $\sigma=\rho^j\Delta\sigma$, compute $\Pi_{\Delta_{n-1}}(\bm{y}-\sigma \bm{a})$ and $r = \psi(\sigma)$.
			\If{$r=0$}
			\State stop,  and $\Pi_{\mathcal{C}}(\bm{y}) = \Pi_{\Delta_{n-1}}(\bm{y}-\sigma \bm{a})$.
			\ElsIf{$r<0$}
			\State set $\sigma_u=\sigma, r_u=r,$ and go to step 18.
			\ElsIf{$r>0$}
			\State set $\sigma_l=\sigma, r_l=r$ and $j = j+1$.
			\EndIf
			\EndFor
			\State{\bf{Secant Phase Begins:}} 
			The approximate solution $\hat{\sigma}\in[\sigma_l,\sigma_u]$ for $\psi(\sigma)=0$ can be obtained by Algorithm \ref{al:2}. Then, $\Pi_{\mathcal{C}}(\bm{y}) = \Pi_{\Delta_{n-1}}(\bm{y}-\hat{\sigma} \bm{a}).$
		\end{algorithmic}
	\end{algorithm}
	\begin{algorithm}[H]
		\caption{A modified secant algorithm for  $\psi(\sigma)=0$}\label{al:2}	
		\begin{algorithmic}[1]
			\Require Given the initial value $\sigma_l, \sigma_u>0$ with $\psi(\sigma_l)>0, \psi(\sigma_u)<0$ and tolerance $\epsilon>0$.
			\State Set ${r}_l = \psi(\sigma_l), r_u = \psi(\sigma_u)$, $s = 1-r_l/r_u, \sigma = \sigma_u-(\sigma_u-\sigma_l)/s$.  Compute $\Pi_{\Delta_{n-1}}(\bm{y}-\sigma \bm{a})$ and $\psi(\sigma)$, set $r = \psi(\sigma)$.
			\While{$|r|>\epsilon$} 
			\If{$r<0$}
			\If{$s\leq2$}
			\State update $\sigma_u=\sigma, r_u=r, s= 1-r_l/r_u, \sigma=\sigma_u-(\sigma_u-\sigma_l)/s$;
			\Else
			\State update $s =\max(r_u/r-1,0.1), \Delta\sigma=(\sigma_u-\sigma)/s,\sigma_u=\sigma,r_u=r$,
			\State	\quad \quad\quad\  $\sigma=\max(\sigma_u-\Delta \sigma,0.6\sigma_l+0.4\sigma_u), s = (\sigma_u-\sigma_l)/(\sigma_u-\sigma). $
			\EndIf
			\Else
			\If{$s\geq2$}
			\State update $\sigma_l=\sigma, r_l=r, s = 1-r_l/r_u,\sigma=\sigma_u-(\sigma_u-\sigma_l)/s.$
			\Else 
			\State update $s = \max(r_l/r-1,0.1), \Delta\sigma =(\sigma-\sigma_l)/s,\sigma_l=\sigma,r_l = r,$
			\State \quad \quad\quad\ $\sigma = \min(\sigma_l+\Delta\sigma,0.6\sigma_u+0.4\sigma_l),s = (\sigma_u-\sigma_l)/(\sigma_u-\sigma).$
			\EndIf
			\EndIf
			\State Compute $\Pi_{\Delta_{n-1}}(\bm{y}-\sigma \bm{a})$ and $\psi(\sigma)$, set $r = \psi(\sigma).$
			\EndWhile\\
			\Return The approximate solution $\hat{\sigma}:=\sigma.$
		\end{algorithmic}
	\end{algorithm}
	In particular, for the second phase of Algorithm \ref{al:1}, we refer to the tailored secant algorithm proposed by Dai and Fletcher \cite{D.F.2006}  for solving  $\psi(\sigma)=0$. The algorithmic framework of  modified secant algorithm is outlined in Algorithm \ref{al:2}. We briefly describe the procedure of Algorithm \ref{al:2}. Based on the initial points $\sigma_l$ and $\sigma_u$ with $\psi(\sigma_l)>0$ and $\psi(\sigma_u)<0$, a new iterative point $\sigma$ is generated by secant method.  
   Now, let us analyze the case where  $r<0$. If $s\leq2$, i.e.,  the interval length of $\left[\sigma_l,\sigma\right]$ is less than $\frac{1}{2}(\sigma_u-\sigma_l)$, then the next iteration proceeds with a secant step using $\sigma_l$ and $\sigma$ as the basis.
  If  $s\geq2$, i.e., the interval length of $\left[\sigma_l,\sigma\right]$ is greater than $\frac{1}{2}(\sigma_u-\sigma_l)$, then  either a secant step based on  $\sigma$ and $\sigma_u$, or a step to the point $0.6\sigma_l+0.4\sigma$ is taken, whichever is the smaller. The modifications in Steps 7 and 8 accelerate the global convergence of the algorithm by shortening the length of the interval  $\left[\sigma_l,\sigma\right]$ by a factor of $0.6$ or less. A similar approach is employed for the case where $r>0$.
	
	Next, we present the convergence results of  Algorithm \ref{al:2} developed in \cite{RR.T.1998, P.Q.S.1998, L.X.L.2022}.
	\begin{theorem}\label{globalc}
		(Global convergence) Suppose that problem \eqref{1.1} is feasible. Let $\{\sigma_i\}$ be the infinite sequence generated by Algorithm \ref{al:2}. Then, $\{\sigma_i\}$ converges to a unique zero point $\sigma^*$ of $\psi$.
	\end{theorem}
	\begin{theorem}\label{localc}
		(Local convergence rate) Let $\{\sigma_i\}$ be the infinite sequence generated by Algorithm \ref{al:2}. Denote $\sigma^*$ as the zero point of $\psi$. Then, $\{\sigma_i\}$ is $3$-step Q-superlinear convergent to $\sigma^*$ in the sense that
		$$
		|\sigma_{i+3}-\sigma^*|=o(|\sigma_i-\sigma^*|).
		$$
	\end{theorem}
	
	\begin{remark}
		In Algorithm \ref{al:1}, the bisection method can also be used to solve  $\psi(\sigma)=0$. The algorithm that combines the bisection method with the bracketing phase  to find the projection $\Pi_{\mathcal{C}}(\cdot)$ is called PBA. From the existing literature \cite{L.L.2017}, we  found that the secant method is more efficient than the bisection method in calculating some projections. We have  tried to apply PBA to  compute $\Pi_{\mathcal{C}}(\cdot)$ and found that  numerical performance of LRSA is superior to that of PBA in practice. Therefore,  we did not consider using the bisection method to solve $\psi(\sigma)=0$ in this paper.
	\end{remark}
	
	\subsection{An algorithm based on   semismooth Newton method}\label{sect:4}
	Since the function $\psi$ is piecewise affine,   $\psi$ is strongly semismooth. Therefore, we consider developing an algorithm based on the semismooth Newton method for problem \eqref{1.1} from the perspective of Lagrangian duality, in which the semismooth Newton method is utilized to solve the equation $\psi(\sigma)=0$.  Furthermore, we show the  convergence results for the  semismooth Newton method.
	\subsubsection{The generalized differential}
	In this subsection, we characterize the generalized differential of the function  $\psi$, which is an  important ingredient in semismooth Newton method.
	
	Recalling that the function $\psi(\cdot)$ is globally Lipschitz continuous on $\mathbb{R}_{+}$, by virtue of  Rademacher's Theorem,  one knows that $\psi(\cdot)$ is almost everywhere Fr{\'e}chet-differentiable. We define the generalized Clarke's differential of $\psi(\cdot)$ as follows:
	$$
	\partial \psi(\sigma):={\rm{conv}}\left(\left\{\lim\limits_{k\to\infty}\psi'(\sigma_k):\sigma_k\to \sigma\ {\rm{such\ that\ }} \psi'(\sigma_k) \ {\rm{is \ well \ defined}} \right\}\right),
	$$
	where `conv' denotes the convex hull. 
	
	Denote $\psi'_{+}, \psi'_{-}$ as the right and left derivatives of $\psi(\cdot)$ at  $\sigma>0$ respectively, i.e.,
	\begin{equation}\label{def}
		\psi'_{+}(\sigma):=\lim\limits_{t\downarrow 0}\frac{\psi(\sigma+t)-\psi(\sigma)}{t},\ 
		\psi'_{-}(\sigma):=\lim\limits_{t\uparrow 0}\frac{\psi(\sigma+t)-\psi(\sigma)}{t}.
	\end{equation}
	Then, it follows  from \cite{C.F.H.1983} that the generalized Clarke's differential of $\psi(\cdot)$ at any given $\sigma>0$ is characterized by
	\begin{equation}\label{psi}
		\partial \psi(\sigma)=\{\alpha \psi'_{+}(\sigma)+(1-\alpha)\psi'_{-}(\sigma)\ | \ \alpha\in\left[0,1\right] \}.
	\end{equation}
	To obtain explicit expressions for $\partial\psi(\sigma)$, we define the following three index subsets of $\{1,\dots,n\}$:
	\begin{equation*}
		\begin{aligned}
			\gamma_1(\sigma):&=\left\{i \mid (\bm{y}-\sigma \bm{a})_i-\frac{\sum_{j=1}^{ \bar{K}(\sigma)}(\bm{y}-\sigma \bm{a})_{\kappa^{\sigma}(j)}-1}{\bar{K}(\sigma)}>0\right\},\\
			\gamma_2(\sigma):&=\left\{i \mid (\bm{y}-\sigma \bm{a})_i-\frac{\sum_{j=1}^{\bar{K}(\sigma)}(\bm{y}-\sigma \bm{a})_{\kappa^{\sigma}(j)}-1}{\bar{K}(\sigma)}=0\right\},\\
			\gamma_3(\sigma):&=\left\{i \mid (\bm{y}-\sigma \bm{a})_i-\frac{\sum_{j=1}^{ \bar{K}(\sigma)}(\bm{y}-\sigma \bm{a})_{\kappa^{\sigma}(j)}-1}{\bar{K}(\sigma)}<0\right\}.
		\end{aligned}
	\end{equation*}
	Combining the definition of $\psi'_{+}$ and $\psi'_{-}$ with \eqref{xstart}, we discuss the following two cases:
	\begin{itemize}
		\item [(1)]
	 If $\gamma_2(\sigma)=\emptyset$, then the right derivative of $\psi$ is computed by
	\begin{equation}
	\begin{aligned}
		\psi'_{+}(\sigma)&=\lim\limits_{t\downarrow0}\frac{\sum_{i=1}^{n}{\bm{a}}_i\left[ \bar{\bm{x}}_i(\sigma+t)-\bar{\bm{x}}_i(\sigma)\right]}{t}\\
		&=\lim\limits_{t\downarrow0}\frac{\sum_{i\in\gamma_1(\sigma)}{\bm{a}}_i\left[ \bar{\bm{x}}_i(\sigma+t)-\bar{\bm{x}}_i(\sigma)\right]}{t}\\
		& = \lim\limits_{t\downarrow0}\frac{\sum_{i\in\gamma_1(\sigma)}{\bm{a}}_i\left[	-t\bm{a}_i-\frac{\sum_{j=1}^{ \bar{K}(\sigma+t)}(\bm{y}-(\sigma+t)\bm{a})_{\kappa^{\sigma}(j)}-1}{\bar{K}(\sigma+t)}+\frac{\sum_{j=1}^{ \bar{K}(\sigma)}(\bm{y}-\sigma \bm{a})_{\kappa^{\sigma}(j)}-1}{\bar{K}(\sigma)}\right]}{t}.
	\end{aligned}
\end{equation}
Since 
\begin{equation}\nonumber
\bar{K}(\sigma+t)=\max\left\lbrace 1 \leq j\leq n\mid ({\bm{y}}-(\sigma+t) {\bm{a}})_{\kappa^{\sigma}(j)}+\frac{1}{j}\left(  1-\sum_{i=1}^{j}({\bm{y}}-(\sigma+t) {\bm{a}})_{\kappa^{\sigma}(i)}\right)>0 \right\rbrace,
\end{equation}
 there always exists a sufficiently small $t$ such that $\bar{K}(\sigma+t)=\bar{K}(\sigma)$. If not, we assume that $\bar{K}(\sigma+t)\neq\bar{K}(\sigma)$ for the sufficiently small $t$. Without loss of generality, we discuss the following two cases:
 \begin{itemize}
 	\item [(i)] If $\bar{K}(\sigma+t)=\bar{K}(\sigma)+1>\bar{K}(\sigma),$ then 
\begin{equation}\label{111}
({\bm{y}}-(\sigma+t) {\bm{a}})_{\kappa^{\sigma}(\bar{K}(\sigma)+1)}+\frac{1}{\bar{K}(\sigma)+1}\left(  1-\sum_{i=1}^{\bar{K}(\sigma)+1}({\bm{y}}-(\sigma+t) {\bm{a}})_{\kappa^{\sigma}(i)}\right)>0.
\end{equation}
On the other hand, according to $\gamma_2(\sigma)=\emptyset$ and the definition of $\bar{K}(\sigma)$, we have
$$
({\bm{y}}-\sigma {\bm{a}})_{\kappa^{\sigma}(\bar{K}(\sigma)+1)}+\frac{1}{\bar{K}(\sigma)+1}\left(  1-\sum_{i=1}^{\bar{K}(\sigma)+1}({\bm{y}}-\sigma {\bm{a}})_{\kappa^{\sigma}(i)}\right)<0,
$$
which implies that there exists a sufficiently small $t$ such that 
$$
\begin{aligned}
({\bm{y}}-\sigma {\bm{a}})_{\kappa^{\sigma}(\bar{K}(\sigma)+1)}&+\frac{1}{\bar{K}(\sigma)+1}\left(  1-\sum_{i=1}^{\bar{K}(\sigma)+1}({\bm{y}}-\sigma {\bm{a}})_{\kappa^{\sigma}(i)}\right)\\
&+t\left(-{{\bm{a}}}_{\kappa^{\sigma}(\bar{K}(\sigma)+1)}+\frac{1}{\bar{K}(\sigma)+1}\sum_{i=1}^{\bar{K}(\sigma)+1}{{\bm{a}}}_{\kappa^{\sigma}(i)}\right)<0,
\end{aligned}
$$ 
which is obviously inconsistent with \eqref{111}.
\item[(ii)] If $\bar{K}(\sigma+t)=\bar{K}(\sigma)-1<\bar{K}(\sigma),$ then
\begin{equation}\nonumber
({\bm{y}}-(\sigma+t) {\bm{a}})_{\kappa^{\sigma}(\bar{K}(\sigma)-1)}+\frac{1}{\bar{K}(\sigma)-1}\left(  1-\sum_{i=1}^{\bar{K}(\sigma)-1}({\bm{y}}-(\sigma+t) {\bm{a}})_{\kappa^{\sigma}(i)}\right)>0.
\end{equation}
However, by definition of $\bar{K}(\sigma)$, one obtains
$$
({\bm{y}}-\sigma {\bm{a}})_{\kappa^{\sigma}(\bar{K}(\sigma))}+\frac{1}{\bar{K}(\sigma)}\left(  1-\sum_{i=1}^{\bar{K}(\sigma)}({\bm{y}}-\sigma {\bm{a}})_{\kappa^{\sigma}(i)}\right)>0.
$$
Thus,  there exists a sufficiently small $t$ such that 
$$
\begin{aligned}
	({\bm{y}}-\sigma {\bm{a}})_{\kappa^{\sigma}(\bar{K}(\sigma))}&+\frac{1}{\bar{K}(\sigma)}\left(  1-\sum_{i=1}^{\bar{K}(\sigma)}({\bm{y}}-\sigma {\bm{a}})_{\kappa^{\sigma}(i)}\right)\\
	&+t\left(-{{\bm{a}}}_{\kappa^{\sigma}(\bar{K}(\sigma))}+\frac{1}{\bar{K}(\sigma)}\sum_{i=1}^{\bar{K}(\sigma)}{{\bm{a}}}_{\kappa^{\sigma}(i)}\right)>0.
\end{aligned}
$$
This indicated that $\bar{K}(\sigma+t)=\bar{K}(\sigma)$ for the sufficiently small $t$, which obviously contradicts to the assumption $\bar{K}(\sigma+t)<\bar{K}(\sigma).$
\end{itemize}

 By combining the above analysis with a simple calculation, one  has
	\begin{equation}\label{rde1}
			\psi'_{+}(\sigma)=\sum_{i\in\gamma_1(\sigma)}\bm{a}_i\left(-{{\bm{a}}}_i+\frac{1}{\bar{K}(\sigma)}\sum_{j=1}^{{\bar{K}(\sigma)}}{{\bm{a}}}_{\kappa^{\sigma}(j)}\right)= \sum_{i\in\gamma_1(\sigma)}\bm{a}_i\left(-\bm{a}_i+\frac{\sum_{j\in\gamma_1(\sigma)}\bm{a}_j}{|\gamma_1(\sigma)|}\right).
	\end{equation}
Similarly, we obtain the left derivative of $\psi$:
	\begin{equation}\label{lde1}
	\psi'_{-}(\sigma)=\sum_{i\in\gamma_1(\sigma)}\bm{a}_i\left(-{{\bm{a}}}_i+\frac{1}{\bar{K}(\sigma)}\sum_{j=1}^{{\bar{K}(\sigma)}}{{\bm{a}}}_{\kappa^{\sigma}(j)}\right)= \sum_{i\in\gamma_1(\sigma)}\bm{a}_i\left(-\bm{a}_i+\frac{\sum_{j\in\gamma_1(\sigma)}\bm{a}_j}{|\gamma_1(\sigma)|}\right).
\end{equation}

	By Cauchy inequality, we  know that $\psi'_{+}(\sigma)$ and $\psi'_{-}(\sigma)$ are non-positive.
	
	\item[(2)] If $\gamma_2(\sigma)\neq\emptyset$, then
	$$
	({\bm{y}}-\sigma {\bm{a}})_{\kappa^{\sigma}(\bar{K}(\sigma)+1)}=	({\bm{y}}-\sigma {\bm{a}})_{\kappa^{\sigma}(\bar{K}(\sigma)+2)}=\dots=	({\bm{y}}-\sigma {\bm{a}})_{\kappa^{\sigma}(\bar{K}(\sigma)+|\gamma_2(\sigma)|)}
	$$
	and 
	$$
	({\bm{y}}-\sigma {\bm{a}})_{\kappa^{\sigma}(\bar{K}(\sigma)+|\gamma_2(\sigma)|)}>({\bm{y}}-\sigma {\bm{a}})_{\kappa^{\sigma}(\bar{K}(\sigma)+|\gamma_2(\sigma)|+1)},
	$$
	 where $|\gamma_2(\sigma)|$ is the cardinality of $\gamma_2(\sigma)$.
	Let $\kappa_{+}^{\sigma}:\{1,2,\dots,|\gamma_2(\sigma)|\}\to\{\kappa^{\sigma}(\bar{K}(\sigma)+1),\kappa^{\sigma}(\bar{K}(\sigma)+2),\dots,\kappa^{\sigma}(\bar{K}(\sigma)+|\gamma_2(\sigma)|)\}$ be permutation such that 
	$$
	\left[ -{\bm{a}}\right] _{\kappa_{+}^{\sigma}(1)}\geq	\left[ -{\bm{a}}\right] _{\kappa_{+}^{\sigma}(2)}\geq\cdots\geq	\left[ -{\bm{a}}\right] _{\kappa_{+}^{\sigma}(|\gamma_2(\sigma)|)}.
	$$

	Next, we proceed to derive  the right derivative of $\psi$. Let $ \tilde{\lambda}_{+}(\sigma)$ be the largest non-negative integer $i\in\{1,\dots,|\gamma_2(\sigma)|\}$ such that
	$$
	({\bm{y}}-(\sigma+t) {\bm{a}})_{\kappa_{+}^{\sigma}(i)}+\frac{1}{\bar{K}(\sigma)+i}\left(  1-\sum_{j=1}^{\bar{K}(\sigma)}({\bm{y}}-(\sigma+t) {\bm{a}})_{\kappa^{\sigma}(j)}-\sum_{j=1}^{i}({\bm{y}}-(\sigma+t) {\bm{a}})_{\kappa_{+}^{\sigma}(j)}\right)>0.
	$$
	Since it holds that for any $i=1,\dots,|\gamma_2(\sigma)|$,
	$$
	({\bm{y}}-\sigma {\bm{a}})_{\kappa_{+}^{\sigma}(i)}+\frac{1}{\bar{K}(\sigma)+i}\left(  1-\sum_{j=1}^{\bar{K}(\sigma)}({\bm{y}}-\sigma {\bm{a}})_{\kappa^{\sigma}(j)}-\sum_{j=1}^{i}({\bm{y}}-\sigma {\bm{a}})_{\kappa_{+}^{\sigma}(j)}\right)=0,
	$$
	one easily knows that $ \tilde{\lambda}_{+}(\sigma)$ is the largest non-negative integer $i\in\{1,\dots,|\gamma_2(\sigma)|\}$ such that
	$$
	- {\bm{a}}_{\kappa_{+}^{\sigma}(i)}+\frac{1}{\bar{K}(\sigma)+i}\left( \sum_{j=1}^{\bar{K}(\sigma)}{\bm{a}}_{\kappa^{\sigma}(j)}+\sum_{j=1}^{i}{\bm{a}}_{\kappa_{+}^{\sigma}(j)}\right) >0.
	$$

Define $\tilde{\zeta}_{+}:=\left[ \sum_{j=1}^{\bar{K}(\sigma)}{\bm{a}}_{\kappa^{\sigma}(j)}+\sum_{j=1}^{\tilde{\lambda}_{+}(\sigma)}{\bm{a}}_{\kappa_{+}^{\sigma}(j)}\right] /(\bar{K}(\sigma)+\tilde{\lambda}_{+}(\sigma))$. By definition of the right derivative, we  obtain that
\begin{equation}\label{rde2}
	\psi'_{+}(\sigma)= \sum_{i\in\gamma_2(\sigma)}\bm{a}_i\max\left(-\bm{a}_i+\tilde{\zeta}_{+},0\right)+\sum_{i\in\gamma_1(\sigma)}\bm{a}_i\left(-\bm{a}_i+\tilde{\zeta}_{+}\right).
\end{equation}
Similarly, we derive the left derivative of $\psi$. 
Let $ \tilde{\lambda}_{-}(\sigma)$ be the smallest non-negative integer $i\in \{|\gamma_2(\sigma)|,|\gamma_2(\sigma)|-1,\dots,1\}$ such that
$$
\begin{aligned}
({\bm{y}}-&(\sigma+t) {\bm{a}})_{\kappa_{+}^{\sigma}(i)}\\
&+\frac{1}{\bar{K}(\sigma)+|\gamma_2(\sigma)|+1-i}\left(  1-\sum_{j=1}^{\bar{K}(\sigma)}({\bm{y}}-(\sigma+t) {\bm{a}})_{\kappa^{\sigma}(j)}-\sum_{j=|\gamma_2(\sigma)|}^{i}({\bm{y}}-(\sigma+t) {\bm{a}})_{\kappa_{+}^{\sigma}(j)}\right)>0.
\end{aligned}
$$
Since for any $i\in \{|\gamma_2(\sigma)|,|\gamma_2(\sigma)|-1,\dots,1\}$,
$$
({\bm{y}}-\sigma {\bm{a}})_{\kappa_{+}^{\sigma}(i)}+\frac{1}{\bar{K}(\sigma)+|\gamma_2(\sigma)|+1-i}\left(  1-\sum_{j=1}^{\bar{K}(\sigma)}({\bm{y}}-\sigma {\bm{a}})_{\kappa^{\sigma}(j)}-\sum_{j=|\gamma_2(\sigma)|}^{i}({\bm{y}}-\sigma {\bm{a}})_{\kappa_{+}^{\sigma}(j)}\right)=0, 
$$
it is clear that $ \tilde{\lambda}_{-}(\sigma)$ is the smallest non-negative integer $i\in\{|\gamma_2(\sigma)|,|\gamma_2(\sigma)|-1,\dots,1\}$ such that
$$
- {\bm{a}}_{\kappa_{+}^{\sigma}(i)}+\frac{1}{\bar{K}(\sigma)+|\gamma_2(\sigma)|+1-i}\left( \sum_{j=1}^{\bar{K}(\sigma)}{\bm{a}}_{\kappa^{\sigma}(j)}+\sum_{j=|\gamma_2(\sigma)|}^{i}{\bm{a}}_{\kappa_{+}^{\sigma}(j)}\right) <0.
$$

Define $\tilde{\zeta}_{-}:=\left[ \sum_{j=1}^{\bar{K}(\sigma)}{\bm{a}}_{\kappa^{\sigma}(j)}+\sum_{j=|\gamma_2(\sigma)|}^{\tilde{\lambda}_{-}(\sigma)}{\bm{a}}_{\kappa_{+}^{\sigma}(j)}\right] /(\bar{K}(\sigma)+|\gamma_2(\sigma)|+1-\tilde{\lambda}_{-}(\sigma))$. By definition of the left derivative,  $\psi'_{-}(\sigma)$ admits the following form:
	\begin{equation}\label{lde2}
	\psi'_{-}(\sigma)= \sum_{i\in\gamma_2(\sigma)}\bm{a}_i\min\left(-\bm{a}_i+\tilde{\zeta}_{-},0\right)+\sum_{i\in\gamma_1(\sigma)}\left(-\bm{a}_i+\tilde{\zeta}_{-}\right).
\end{equation}
\end{itemize}

It follows from the Cauchy inequality that  $\psi'_{+}(\sigma)$ is non-positive.
	For the subsequent algorithm design, we take $\alpha=1$ in \eqref{psi}, which also ensures that the elements chosen from $\partial\psi(\cdot)$ are all non-positive.
	
	\subsubsection{ Algorithm description}
	In this subsection, we present an efficient algorithm based on semismooth Newton method (SSN) for computing the projection $\Pi_{\mathcal{C}}(\cdot)$. The algorithm framework is shown in Algorithm \ref{al:3}.  In Algorithm \ref{al:3}, given a vector $\bm{y}\in\mathbb{R}^n$, we initialize $\sigma=0$ and compute $\psi(0)=\bm{a}^{\top}\Pi_{\Delta_{n-1}}(\bm{y})-b$. If $\psi(0)\leq 0$, then the projection  $\Pi_{\mathcal{C}}(\bm{y})$ is obtained; otherwise, we apply the semismooth Newton  algorithm to solve $\psi(\sigma)=0$ and obtain the approximate optimal solution. Finally, we compute the projection $\Pi_{\mathcal{C}}(\bm{y}).$ 
	
	\begin{algorithm}[H]
		\caption{An algorithm based on semismooth Newton method for  $\Pi_{\mathcal{C}}(\cdot)$}\label{al:3}	
		\begin{algorithmic}[1]
			\Require Given the vector $\bm{y}\in\mathbb{R}^n$, set $\sigma^0\in[0,+\infty),\hat{\mu}\in(0,1/2),\hat{\delta},\hat{\tau}_1\in(0,1),\hat{\tau}_2\in(0,1], \epsilon>0$, and $j=0$.
			\State  Compute $\Pi_{\Delta_{n-1}}(\bm{y})$ and $\psi(0)$. If $\psi(0)\leq0$, then $\Pi_{\mathcal{C}}(\bm{y}) = \Pi_{\Delta_{n-1}}(\bm{y})$; otherwise, go to  step 2.
			\State  {\bf{Semismooth Newton Algorithm Begins:}} 
			\While{$|\psi(\sigma^j)|>\epsilon$} 
			\State  Choose $\upsilon_j\in\partial \psi(\sigma^j)$, then compute the Newton direction via $$
			\Delta\sigma^j=-\psi(\sigma^j)/(\upsilon_j-\bar{\epsilon}_j),
			$$
			where  $\bar{\epsilon}_j:=\hat{\tau}_2\min\{\hat{\tau}_1,|\psi(\sigma^j)|\}.$
			\State  Let $m_j$ be the smallest non-negative integer $m$ such that  
			$$
			h(\sigma^j+\hat{\delta}^m\Delta\sigma^j)\geq h(\sigma^j)+\hat{\mu}\hat{\delta}^m\psi(\sigma^j)\Delta\sigma^j.
			$$
			\State  Update $\sigma^{j+1}=\sigma^{j}+\hat{\delta}^{m_j}\Delta\sigma^j,\ j\gets j+1$.
			\EndWhile
			\State Compute the projection $\Pi_{\mathcal{C}}(\bm{y}) = \Pi_{\Delta_{n-1}}(\bm{y}-\sigma^{j} \bm{a})$.
		\end{algorithmic}
	\end{algorithm}
	From  \cite[Theorem 3.1]{W.L.L.2022} and \cite{Z.S.T.2010}, we establish  the convergence results for Algorithm \ref{al:3}, which  are stated in the next theorem.
	\begin{theorem}\label{thm:1}
		Let $\{\sigma^j\}$ be the infinite sequence generated by Algorithm \ref{al:3}.  Assume that  given data vectors ${\bm{a}}\in\mathbb{R}^n$ and $b\in\mathbb{R}$ satisfy ${\bm{a}}\neq b{\bm{e}}$, where ${\bm{e}}$ denotes the vector with all entries being $1$.
		Then $\{\sigma^j\}$ is  bounded and any accumulation point $\sigma^*$ of $\{\sigma^j\}$ is an optimal solution to problem \eqref{dual}. In addition, the rate of local convergence is quadratic, i.e., $|\sigma^{j+1}-\sigma^*|=O(|\sigma^j-\sigma^*|^2)$.
	\end{theorem}
	\begin{proof}
		By \cite[Proposition 3.3]{Z.S.T.2010}, for any $j\geq 0$, if $\psi(\sigma^j)\neq0,$ then $\Delta\sigma^j$ is an ascent direction. Thus, Algorithm \ref{al:3} is well defined. Since problem \eqref{dual} satisfies the  Slater condition, we know from Proposition \ref{pro7} that the function $h$ is coercive. It is known from \cite{DB.2009} that the function $h$ is coercive if and only if for  every $\tilde{\alpha}\in\mathbb{R}$ the set $\left\lbrace \sigma\mid h(\sigma)\leq \tilde{\alpha}\right\rbrace$ is compact. Therefore the sequence $\{\sigma^j\}$ is bounded. Let $\sigma^*$ be any cluster point of $\{\sigma^j\}$. From \cite[Theorem 3.4]{Z.S.T.2010} and the  concavity of $h$,  one can readily get that $\sigma^*$ is the optimal solution to the problem \eqref{dual}.
		
		Next, we verify that any $\upsilon^0\in\partial \psi(\sigma^*)$ is negative. Suppose by contradiction that $\upsilon^0=0.$ By the definition of $\partial\psi(\cdot)$ in \eqref{psi},  we know that if an element of $\partial\psi(\sigma^*)$ is equal to $0$ then  $ \bm{a}_i=\bm{a}_j,\ \forall i,j\in\gamma_1(\sigma^*)$. Let $ \bm{a}_i=\bm{a}_j=\bar{a},\ \forall i,j\in\gamma_1(\sigma^*)$. Since $\sigma^*$ is an optimal solution to problem \eqref{dual},  we have 
		$$
		\psi(\sigma^*)=\bm{a}^{\top}\Pi_{\Delta_{n-1}}(\bm{y}-\sigma^*\bm{a})-b=0,
		$$
		which implies 
		\begin{equation}\label{ab}
			\begin{aligned}
				&\sum_{i\in\gamma_1(\sigma^*)\cup\gamma_2(\sigma^*)}\bm{a}_i\left[ (\bm{y}-\sigma^* \bm{a})_i-\frac{\sum_{j=1}^{{\bar{K}}(\sigma^*)}(\bm{y}-\sigma^* \bm{a})_{\kappa^{\sigma^*}(j)}-1}{\bar{K}(\sigma^*)}\right]-b\\
				&=\bar{{a}}\sum_{i\in\gamma_1(\sigma^*)}(\bm{y}-\sigma^* \bm{a})_i-\bar{{a}}\sum_{j=1}^{\bar{K}(\sigma^*)}(\bm{y}-\sigma^* \bm{a})_{\kappa^{\sigma^*}(j)}+\bar{{a}}-b\\
				&=\bar{{a}}\sum_{j=1}^{\bar{K}(\sigma^*)}(\bm{y}-\sigma^* \bm{a})_{\kappa^{\sigma}(j)}-\bar{{a}}\sum_{j=1}^{\bar{K}(\sigma^*)}(\bm{y}-\sigma^* \bm{a})_{\kappa^{\sigma^*}(j)}+\bar{{a}}-b=\bar{{a}}-b=0,
			\end{aligned}
		\end{equation}
		which contradicts to our assumption. Thus, $\upsilon^0$ is negative.
		
		Since any $\upsilon^0\in\partial \psi(\sigma^*)$ is negative, $\{(\upsilon_j-\bar{\epsilon}_j)^{-1}\}$ is uniformly bounded for all $j$ sufficiently large.   Recall that $\psi(\cdot)$ is strongly semismooth. Thus,  by virtue of  \cite[Theorem 3.5]{Z.S.T.2010},  it holds that for all $j$ sufficiently large, 
		\begin{equation}\label{qua}
			|\sigma^j+\Delta\sigma^j-\sigma^*|={O}(|\sigma^j-\sigma^*|^2),
		\end{equation}
		and there exists a constant $\tilde{\delta}>0$ such that $\psi(\sigma^j)\Delta\sigma^j\geq\tilde{\delta}|\Delta\sigma^j|^2. $
		Moreover, it follows from \cite[Theorem 3.3 \& Remark 3.4]{F.1995} that for $\hat{\mu}\in(0,1/2)$, there exists an integer $j_0$ such that for any $j\geq j_0$, $h(\sigma^j+\Delta\sigma^j)\geq h(\sigma^j)+\hat{\mu}\psi(\sigma^j)\Delta\sigma^j,$ which implies that for all $j\geq j_0$, 
		\begin{equation}\label{sigmaj}
			\sigma^{j+1}=\sigma^j+\Delta\sigma^j.
		\end{equation}
		Combining with \eqref{qua} and \eqref{sigmaj}, we complete the proof.
	\end{proof}
	
	\section{Numerical experiments}\label{sect:5}
	In this section, we compare   LRSA and SSN  with the solver Gurobi for computing the projection onto the intersection of simplex and singly linear constraint. All our experiments are conducted in MATLAB R2019a  on a Dell desktop  computer with Intel Xeon  Gold 6144 CPU @ 3.50GHz  and 256 GB RAM. The code for calculating the projection onto the simplex is freely available at \url{https://lcondat.github.io/download/condat_simplexproj.c}

	In our experiments, the parameters in Algorithm LRSA are set to be $\Delta\sigma=1, \rho=2$, and $\epsilon=10^{-7}$. The parameters in Algorithm SSN are chosen as $\sigma^0=0,\hat{\delta}=0.5,\hat{\tau}_1=1$ and $\hat{\tau}_2=10^{-3}$.  For this  accuracy tolerance $\epsilon$, we  terminate  LRSA and SSN when $\psi(0)\leq0$, or $|\psi(\hat{\sigma})|\leq \epsilon$ or the number of iterations exceeds $500$, where $|\psi(\hat{\sigma})|$ represents the absolute value of $\psi(\cdot)$ at  the approximate solution $\hat{\sigma}$.  We test instances with $n= 5\times10^{4}, 10^{5}, 5\times10^{5}, 10^{6}, 5\times10^{6}, 10^{7}, 5\times10^{7}$, and $10^{8}$, respectively. To make the results more convincing, each instance is run $5$ times.

	\begin{table}[h]
		\caption{Numerical results of the LRSA, SSN, and Gurobi for Example \ref{ex1} on random data }\label{tab:1}%
		\begin{tabular*}{\textwidth}{@{\extracolsep\fill}ccccc}
			\hline
			$n$ &  Algorithm  & avgtime & iter & {$|\psi(\hat{\sigma})|$} \\
			\midrule
			 5e+04	 & LRSA	 & \bf{0.0030} 	 &  8 	 & 5.1e-13 \\  
			& SSN	 & 0.0154 	 &  7 	 & 3.8e-08 \\  
			& Gurobi	 & 0.1484 	 &  16 	 & 1.8e-10 \\  
			1e+05	 & LRSA	 & \bf{0.0094}	 &  10 	 & 1.7e-12 \\  
			& SSN	 & 0.0190 	 &  7 	 & 4.7e-08 \\  
			& Gurobi	 & 0.2996 	 &  16 	 & 3.8e-10 \\  
			5e+05	 & LRSA	 & 0.0312 	 &  10 	 & 2.3e-12 \\  
			& SSN	 & \bf{0.0252} 	 &  4 	 & 3.7e-08 \\  
			& Gurobi	 & 1.8182 	 &  17 	 & 2.5e-10 \\  
			1e+06	 & LRSA	 & \bf{0.0468} 	 &  10 	 & 1.4e-11 \\  
			& SSN	 & 0.0996 	 &  9 	 & 3.1e-08 \\  
			& Gurobi	 & 4.3182 	 &  20 	 & 3.6e-09 \\  
			5e+06	 & LRSA	 & \bf{0.2284} 	 &  10 	 & 4.1e-10 \\  
			& SSN	 & 0.2310 	 &  4 	 & 6.7e-09 \\  
			& Gurobi	 & 21.4664 	 &  19 	 & 4.4e-09 \\  
			1e+07	 & LRSA	 & 0.4310 	 &  10 	 & 6.0e-11 \\  
			& SSN	 & \bf{0.4224} 	 &  4 	 & 9.3e-09 \\  
			& Gurobi	 & 46.2566 	 &  21 	 & 7.1e-07 \\  
			5e+07	 & LRSA	 & \bf{2.0432} 	 &  10 	 & 1.9e-08 \\  
			& SSN	 & 2.7908 	 &  6 	 & 2.1e-08 \\  
			& Gurobi	 & 222.5594 	 &  21 	 & 1.2e-10 \\  
			1e+08	 & LRSA	 & \bf{4.4368} 	 &  11 	 & 4.2e-10 \\  
			& SSN	 & 4.5172 	 &  5 	 & 2.1e-08 \\  
			& Gurobi	 & 445.3348 	 &  21 	 & 6.1e-09 \\
			\hline
		\end{tabular*}
	{\emph{Note.}}\ The best result in each experiment is highlighted in bold.
	\end{table}

	The general case of problem \eqref{1.1} we consider in numerical experiments  is described in the first example below.
	
	\begin{example}\label{ex1}
		For problem \eqref{1.1},  the vectors $\bm{y}$ and $\bm{a}\in\mathbb{R}^n$ are randomly generated $n\times 1$ vectors with $\bm{y}=-3*{{rand}}(n,1)$ and $\bm{a} = 20*{{rand}}(n,1)$, respectively. To ensure the feasibility of problem \eqref{1.1}, we take $b=0.45*\max(\bm{a}).$ 
	\end{example}

	Table \ref{tab:1} reveals the numerical results of  Algorithm LRSA, Algorithm SSN, and Gurobi for Example \ref{ex1} on random data. This table includes the average CPU time (avgtime) for the five runs, the maximum number of iterations (iter) among the five runs, and the absolute value of $\psi(\cdot)$ at the  solution $\hat{\sigma}$ ($|\psi(\hat{\sigma})|$).
	It is observed that Algorithm LRSA, Algorithm SSN, and Gurobi successfully solve all instances with high accuracy.
	As shown in  Table \ref{tab:1}, the running time of LRSA and SSN is significantly shorter than  that of Gurobi.
	In particular,  when  the dimension is $n=10^{8}$, the time of the LRSA and SSN is less than $5$ seconds, while  Gurobi  needs more than $400$ seconds to obtain the approximate projection. The running time of LRSA  is about $100$ times  faster than that of Gurobi, while the running time of SSN  is about $80$ times  faster than that of Gurobi. 
	These results in Table \ref{tab:1} highlight the excellent performance of the LRSA and SSN in practice.

	In the second example of the numerical experiments, we consider the degenerate case of problem \eqref{1.1}, i.e., the relative interior of $\mathcal{C}$ lies in the relative boundary of the simplex.
	
	\begin{example}\label{ex2}
		For problem \eqref{1.1},  the vectors $\bm{y}$ is randomly generated $n\times 1$ vectors with $\bm{y}=-3*{{rand}}(n,1)$. We set ${\bm{a}} = [b+1,b,\dots,b]\in\mathbb{R}^n$ and $b = 50$. 
	\end{example}
	\begin{table}[h]
	\caption{Numerical results of the LRSA, SSN, and Gurobi for Example \ref{ex2} on random data}\label{tab:2}%
	\begin{tabular*}{\textwidth}{@{\extracolsep\fill}ccccc}
		\hline
		$n$ &  Algorithm  & avgtime & iter & {$|\psi(\hat{\sigma})|$} \\
		\midrule
		5e+04	 & LRSA	 & {0.0032} 	 &  5 	 & 2.6e-09 \\  
		& SSN	 & \bf{0.0030} 	 &  1 	 & 2.3e-09 \\  
		& Gurobi	 & 0.1436 	 &  13 	 & 3.6e-13 \\  
		1e+05	 & LRSA	 & \bf{0.0030} 	 &  3 	 & 6.3e-13 \\  
		& SSN	 & 0.0094 	 &  1 	 & 6.8e-09 \\  
		& Gurobi	 & 0.2862 	 &  14 	 & 4.4e-12 \\  
		5e+05	 & LRSA	 & \bf{0.0126} 	 &  5 	 & 6.3e-08 \\  
		& SSN	 & 0.0254 	 &  1 	 & 3.3e-11 \\  
		& Gurobi	 & 1.7308 	 &  15 	 & 2.0e-12 \\  
		1e+06	 & LRSA	 & \bf{0.0282} 	 &  5 	 & 9.5e-08 \\  
		& SSN	 & 0.0560 	 &  1 	 & 2.4e-08 \\  
		& Gurobi	 & 3.6060 	 &  15 	 & 1.7e-11 \\  
		5e+06	 & LRSA	 & \bf{0.0778} 	 &  3 	 & 1.2e-10 \\  
		& SSN	 & 1.8348 	 &  11 	 & 6.5e-08 \\  
		& Gurobi	 & 18.9408 	 &  15 	 & 1.7e-10 \\  
		1e+07	 & LRSA	 & \bf{0.1314} 	 &  3 	 & 3.5e-10 \\  
		& SSN	 & 1.0902 	 &  2 	 & 7.4e-11 \\  
		& Gurobi	 & 39.5760 	 &  16 	 & 2.6e-11 \\  
		5e+07	 & LRSA	 & \bf{0.6278} 	 &  3 	 & 1.1e-10 \\  
		& SSN	 & 5.1334 	 &  2 	 & 1.2e-11 \\  
		& Gurobi	 & 180.6106 	 &  14 	 & 2.1e-08 \\  
		1e+08	 & LRSA	 & \bf{1.2404} 	 &  3 	 & 8.2e-10 \\  
		& SSN	 & 10.0886 	 &  2 	 & 4.6e-10 \\  
		& Gurobi	 & 418.7790 	 &  17 	 & 2.2e-09 \\
		\hline
	\end{tabular*}
	{\emph{Note.}}\ The best result in each experiment is highlighted in bold.
\end{table}
	
Table \ref{tab:2} presents the numerical results of  Algorithm LRSA, Algorithm SSN, and Gurobi for Example \ref{ex2} on random data. As shown inTable \ref{tab:2}, for degenerate cases, both Algorithm LRSA and Algorithm SSN outperform the solver Gurobi in terms of runtime.
	For degenerate cases in Example \ref{ex2}, since $\psi(\sigma)={\bm{a}}^{\top}\Pi_{\Delta_{n-1}}({\bm{y}}-\sigma{\bm{a}})-b=( \Pi_{\Delta_{n-1}}({\bm{y}}-\sigma{\bm{a}}))_{1}$, $\psi(\sigma)$ is  usually close to zero or a small value at the start of the algorithm. In most cases, after only a few iterations in the Secant Phase, $|\psi(\sigma)|$ already meets the specified tolerance. In degenerate cases, the LRSA algorithm outperforms Gurobi because it is specifically designed to exploit the structural characteristics of the problem. In contrast, Gurobi, while being a cutting-edge optimization solver, is designed with a broad applicability in mind, which may not fully leverage the specialized nuances of certain problems. Meanwhile, by comparing Algorithm LRSA with Algorithm SSN, LRSA demonstrates greater efficiency in handling degenerate cases. Based on the structure of $\bm{a}$ and $b$ in Example \ref{ex2}, Theorem \ref{thm:1} suggests that the condition $\upsilon^0 \in \partial \psi(\sigma^*)$ being negative may not always be met by SSN. During numerical iterations, we observe that SSN's line search requires more iterations, potentially explaining its longer running time.	
	
The third example we consider is a projection problem of the set $\mathcal{C}$ involved in the Wasserstein distributionally robust  portfolio model \cite{Z.L.2023}.
	\begin{example}\label{ex3}
		Assume that $\{\hat{{\bm{\xi}}}_1,\hat{{\bm{\xi}}}_2,\dots,\hat{{\bm{\xi}}}_m\}$ is a set of the independent observations of the asset return ${\bm{\xi}}\in\mathbb{R}^n$. Denote the matrix ${\bm{A}}:=[\tilde{{\bm{\mu}}}-\hat{{\bm{\xi}}}_1,\ldots,\tilde{{\bm{\mu}}}-\hat{{\bm{\xi}}}_m] ^{\top}\in\mathbb{R}^{m\times n}$, where $\tilde{{\bm{\mu}}}=\frac{1}{m}\sum_{i=1}^{m}\hat{{\bm{\xi}}}_i.$ We consider the projection problem involved in \cite{Z.L.2023} as follows:
	\begin{equation}\label{example3}
				\begin{split}
				\mathop{\min}\limits_{\bm{x}\in\mathbb{R}^n}& \ \frac{1}{2}\|\bm{x}-(\tilde{{\bm{u}}}+A^{\top}\tilde{{\bm{v}}})\|^2  \\
				\mbox{s.t.}\  & \ \tilde{\bm{\mu}}^{\top}\bm{x}\geq \hat{\rho},\ {\bm{x}}\in\Delta_{n-1},\\
			\end{split}
	\end{equation}
    where the vectors $\tilde{\bm{u}}\in\mathbb{R}^n$ and $\tilde{\bm{v}}\in\mathbb{R}^m$ are randomly generated  vectors with $\tilde{\bm{u}}={rand}(n,1)$ and $\tilde{\bm{v}} = {{rand}}(m,1)$, respectively,
	the target expected return $\hat{\rho}$ is generated by $\hat{\rho}=\min(\tilde{\bm{\mu}})*rand(1,1)$. We collected some stock return datasets from Ken French's website\footnote{\url{https://mba.tuck.dartmouth.edu/pages/faculty/ken.french/data_library.html}}, which include: (1) $25$  Portfolios Formed on Book-to-Market and Operating Profitability ({\emph{25BEMEOP}}); (2) $100$ Portfolios Formed on Size and Investment ({\emph{100MEINV}}). Due to the limited dimensionality of real datasets and their comparable numerical performance, we only select two real datasets for the numerical experiments.
		\begin{table}[H]
		\caption{Numerical results of the LRSA, SSN, and Gurobi for Example \ref{ex3} on real data }\label{tab:3}%
		\begin{tabular*}{\textwidth}{@{\extracolsep\fill}ccccc}
			\hline
			$n$ &  Algorithm  & avgtime & iter & {$|\psi(\hat{\sigma})|$} \\
			\midrule
		 $25BEMEOP$	 & LRSA	 & \bf{0.0000}	 &  1 	 & 5.9e-01 \\  
		& SSN	 & \bf{0.0000} 	 &  1 	 & 5.9e-01 \\  
		& Gurobi	 & 0.0030 	 &  8 	 & 5.9e-01 \\  
		$100MEINV$	 & LRSA	 & \bf{0.0000} 	 &  1 	 & 9.0e-01 \\  
		& SSN	 & \bf{0.0000}	 &  1 	 & 9.0e-01 \\  
		& Gurobi	 & 0.0032	 &  8 	 & 9.0e-01 \\
			\hline
		\end{tabular*}
		{\emph{Note.}}\ The best result in each experiment is highlighted in bold.
	\end{table}
\end{example}
Numerical results of Algorithm LRSA, Algorithm SSN and Gurobi for Example \ref{ex3} on real data are shown in Table \ref{tab:3}. As we can see, the running times of LRSA and SSN are slightly faster than of Gurobi. The optimal solution to problem \eqref{example3} satisfies conditions $\sigma=0$ and $\psi(0)\leq0$.

	In summary, for the general case of  problem \eqref{1.1}, the numerical performance of Algorithm SSN and Algorithm  LRSA is similar and significantly better than that of  Gurobi. For the degenerate problem \eqref{1.1}, the numerical performance of Algorithm  LRSA is better than that of Algorithm  SSN and Gurobi. Therefore,  we  reasonably conclude that  Algorithm LRSA  is superior to  Algorithm SSN and Gurobi in  computing  $\Pi_{\mathcal{C}}(\cdot)$.

	\section{The generalized Jacobian of the projection $\Pi_{\mathcal{C}}(\cdot)$}\label{sect:6}
	In this section, we follow the principles in \cite{L.S.T.2018,L.S.T.2020,H.S.1997} to derive the generalized HS-Jacobian of the projection onto the intersection of  simplex and singly  linear inequality constraint.
	
	Denote 
	$$
	\bm{B}=\begin{bmatrix}
		-\bm{I}_n\\
		\bm{a}^{\top}
	\end{bmatrix}\in\mathbb{R}^{(n+1)\times n},\quad
	\bm{c}=\begin{bmatrix}
		{\bf{0}}_n\\
		b
	\end{bmatrix}\in\mathbb{R}^{n+1}.
	$$
	Then, problem \eqref{1.1} can be reformulated as
	\begin{equation}\label{re:1}
		\mathop{\min}\limits_{\bm{x}\in\mathbb{R}^n} \left\{ \frac{1}{2}\|\bm{x}-\bm{y}\|^2 \ | \ \bm{e}^{\top}\bm{x}=1, \bm{Bx}\leq \bm{c}\right\}.
	\end{equation}
	Note that for given $\bm{y}\in\mathbb{R}^n$, $\Pi_{\mathcal{C}}(\bm{y})$ is the unique optimal solution to problem \eqref{re:1}. The KKT conditions for problem \eqref{re:1} are formulated as 
	\begin{equation}\label{kkt}
		\left\{
		\begin{aligned}
			&\Pi_{\mathcal{C}}(\bm{y})-\bm{y}-\lambda \bm{e}+\bm{B}^{\top}\bm{\mu}=0, \\
			&\bm{e}^{\top}\Pi_{\mathcal{C}}(\bm{y})=1,\\
			&\bm{B}\Pi_{\mathcal{C}}(\bm{y})-\bm{c}\leq 0,\bm{\mu}\geq 0, \bm{\mu}^{\top}(\bm{B}\Pi_{\mathcal{C}}(\bm{y})-\bm{c})=0
		\end{aligned}
		\right.
	\end{equation}
	with Lagrange multipliers $\lambda\in\mathbb{R}, \bm{\mu}\in\mathbb{R}_{+}^n$.  Define the set of multipliers by
	$$
	\mathcal{M} (\bm{y}):=\{(\lambda,\bm{\mu})\in\mathbb{R}\times\mathbb{R}^n\ |\ (\bm{y},\lambda,\bm{\mu})\ {\rm{satisfies}}\ \eqref{kkt}\}.
	$$
	It is clear that $\mathcal{M}(\bm{y})$ is a nonempty polyhedral set containing no lines. Thus, as stated in \cite[ Corollary 18.5.3]{RR.T.1970}, $\mathcal{M} (\bm{y})$ has at least one extreme point. Let $\mathcal{I}(\bm{y})$ be the active index set:
	\begin{equation}\label{Iy}
		\mathcal{I}(\bm{y}):=\{i\ |\ \bm{B}_i\Pi_{\mathcal{C}}(\bm{y})=\bm{c}_i,\ i=1,\dots,n+1 \},
	\end{equation}
	where $B_i$ is the $i$th row of the matrix $B$.
	We define a collection of index sets by
	$$
	\begin{aligned}
		\mathcal{K}_{\mathcal{C}}(\bm{y}):=&\{K\subseteq\{1,\dots,n+1\}\ |\ \exists(\lambda,\bm{\mu})\in\mathcal{M}(\bm{y})\ \ {\rm{s.t.\ supp}}(\bm{\mu})\subseteq K\subseteq \mathcal{I}(\bm{y}),\\
		& [\bm{B}_{K}^{\top}\ \bm{e}]\ {\rm{is \ of\ full\ column \ rank}}\},
	\end{aligned}
	$$
	where ${\rm{supp}}(\bm{\mu})$  denotes the support of $\bm{\mu}$, i.e., the set of indices $i$ such that $\bm{\mu}_i\neq0$. By the existence of the extreme point of $\mathcal{M}(\bm{y})$, we know that $\mathcal{K}_{\mathcal{C}}(\bm{y})$ is nonempty. According to \cite{H.S.1997}, the generalized HS-Jacobian of $\Pi_{\mathcal{C}}(\cdot)$ at $\bm{y}$ is defined by 
	$$
	\mathcal{N}_{\mathcal{C}}(\bm{y}):=\left\{\bm{N}\in\mathbb{R}^{n\times n}\ | \ \bm{N}=\bm{I}_n- [\bm{B}_{K}^{\top}\ \bm{e}]\left(
	\begin{bmatrix}
		\bm{B}_{K}\\
		\bm{e}^{\top}
	\end{bmatrix}[\bm{B}_{K}^{\top}\ \bm{e}]\right)^{-1}\begin{bmatrix}
		\bm{B}_{K}\\
		\bm{e}^{\top}
	\end{bmatrix},\ K\in \mathcal{K}_{\mathcal{C}}(\bm{y})\right\}.
	$$
	The generalized HS-Jacobian of $\Pi_{\mathcal{C}}(\cdot)$ has some important properties \cite{H.S.1997,L.S.T.2020}, which are summarized in the following propositions.
	\begin{proposition}\label{pro9}
		For any given $\bm{y}\in\mathbb{R}^n$, there exists a neighborhood $\mathcal{Y}$ of $\bm{y}$ such that
		$$
		\mathcal{K}_{\mathcal{C}}(\bm{w})\subseteq\mathcal{K}_{\mathcal{C}}(\bm{y}),\ \mathcal{N}_{\mathcal{C}}(\bm{w})\subseteq\mathcal{N}_{\mathcal{C}}(\bm{y}),\ \forall \bm{w}\in\mathcal{Y},
		$$
		and 
		$$
		\Pi_{\mathcal{C}}(\bm{w})=\Pi_{\mathcal{C}}(\bm{y})+\widehat{\bm{N}}(\bm{w}-\bm{y}),\ \forall \widehat{\bm{N}}\in\mathcal{N}_{\mathcal{C}}(\bm{w}).
		$$
		Denote
		\begin{equation}\label{BY}
			\bm{N}_0:=\bm{I}_n- [\bm{B}_{\mathcal{I}(\bm{y})}^{\top}\ \bm{e}]\left(
			\begin{bmatrix}
				\bm{B}_{\mathcal{I}(\bm{y})}\\
				\bm{e}^{\top}
			\end{bmatrix}[\bm{B}_{\mathcal{I}(\bm{y})}^{\top}\ \bm{e}]\right)^{\dagger}\begin{bmatrix}
				\bm{B}_{\mathcal{I}(\bm{y})}\\
				\bm{e}^{\top}
			\end{bmatrix},
		\end{equation}
		where $\mathcal{I}(\bm{y})$ is defined as in \eqref{Iy}. Then, $\bm{N}_0\in\mathcal{N}_{\mathcal{C}}(\bm{y}).$
	\end{proposition}
	\begin{proposition}\label{pro10}
		Let $\bm{\theta}\in\mathbb{R}^n$ be a given vector with each entry $\bm{\theta}_i$ being $0$ or $1$ for each $i = 1,\dots,n$. Let $\bm{\Theta}={\rm{Diag}}(\bm{\theta})$ and $\bm{\Sigma}=\bm{I}_n-\bm{\Theta}$. For any given matrix $\bm{H}\in\mathbb{R}^{M\times n}$, it holds that 
		$$
		\bm{P}:=\bm{I}_n-[\bm{\Theta}\ \bm{H}^{\top}]\left(
		\begin{bmatrix}
			\bm{\Theta}\\
			\bm{H}
		\end{bmatrix}[\bm{\Theta}\ \bm{H}^{\top}]\right)^{\dagger}\begin{bmatrix}
			\bm{\Theta}\\
			\bm{H}
		\end{bmatrix}=\bm{\Sigma}-\bm{\Sigma} \bm{H}^{\top}(\bm{H}\bm{\Sigma} \bm{H}^{\top})^{\dagger}\bm{H}\bm{\Sigma}.
		$$
	\end{proposition}
	Next, we shall calculate the element $\bm{N}_0$ of the generalized HS-Jacobian matrix by virtue of the effective approach proposed in Proposition \ref{pro9} and Proposition \ref{pro10}.
	
	For given $\bm{y}\in\mathbb{R}^n$, we define the following index subsets of $\{1,\dots,n\}$:
	$$
	\mathcal{K}_1:=\{i\ |\ (\Pi_{\mathcal{C}}(\bm{y}))_i=0\},\ \mathcal{K}_2:=\{1,\dots,n\}\backslash\mathcal{K}_1.
	$$
	It is easy to know from the definition of $\mathcal{C}$ that $|\mathcal{K}_2|\neq0$. And we also have $ |\mathcal{K}_1|+|\mathcal{K}_2|=n.$
	\begin{theorem}
		Assume that $\bm{a}\in\mathbb{R}^n,b\in\mathbb{R}$ in problem \eqref{1.1} are given. For given $\bm{y}\in\mathbb{R}^n,$ denote 
		$$
		\bm{w}_i=\left\{
		\begin{aligned}
			&1,\ i\in\mathcal{K}_2, \\
			&0,\ otherwise,
		\end{aligned}
		\right.
		(\bm{e}_{\mathcal{K}_2}^n)_i=\left\{
		\begin{aligned}
			&1,\ i\in\mathcal{K}_2, \\
			&0,\ otherwise,
		\end{aligned}
		\right.
		(\bm{a}_{\mathcal{K}_2}^n)_i=\left\{
		\begin{aligned}
			&\bm{a}_i,\ i\in\mathcal{K}_2, \\
			&0,\ otherwise,
		\end{aligned}
		\right.  i=1,2,\dots,n.
		$$
		Then, the element $\bm{N}_0$ of the generalized HS-Jacobian
		for $\Pi_{\mathcal{C}}(\cdot)$ at $\bm{y}$ admits the following explicit expressions:
		
		\noindent{\bf{\uppercase\expandafter{\romannumeral1}}}. If $\bm{a}^{\top}\Pi_{\mathcal{C}}(\bm{y})\neq b$, then
		$$
		\bm{N}_0 = {\rm{Diag}}(\bm{w})-\frac{1}{|\mathcal{K}_2|}\bm{e}_{\mathcal{K}_2}^n(\bm{e}_{\mathcal{K}_2}^n)^{\top}.
		$$
		
		\noindent{\bf{\uppercase\expandafter{\romannumeral2}}}. If $\bm{a}^{\top}\Pi_{\mathcal{C}}(\bm{y})= b$, then the following two cases are taken into consideration. Denote 
		$$
		\eta:=\|\bm{a}_{\mathcal{K}_2}\|^2|\mathcal{K}_2|-(\bm{a}_{\mathcal{K}_2}^{\top}\bm{e}_{\mathcal{K}_2})^2.
		$$
		
		\noindent{\bf{(\romannumeral1)}} If $\eta\neq 0$, then
		$$
		\begin{aligned}
			\bm{N}_0=&{\rm{Diag}}(\bm{w})-\frac{1}{\eta}(\sqrt{|\mathcal{K}_2|}\bm{a}_{\mathcal{K}_2}^n-\|\bm{a}_{\mathcal{K}_2}\|\bm{e}_{\mathcal{K}_2}^n)(\sqrt{|\mathcal{K}_2|}\bm{a}_{\mathcal{K}_2}^n-\|\bm{a}_{\mathcal{K}_2}\|\bm{e}_{\mathcal{K}_2}^n)^{\top}\\
			&-\frac{\sqrt{|\mathcal{K}_2|}\|\bm{a}_{\mathcal{K}_2}\|-\bm{a}_{\mathcal{K}_2}^{\top}\bm{e}_{\mathcal{K}_2}}{\eta}(\bm{a}_{\mathcal{K}_2}^n(\bm{e}_{\mathcal{K}_2}^n)^{\top}+\bm{e}_{\mathcal{K}_2}^n(\bm{a}_{\mathcal{K}_2}^n)^{\top}).
		\end{aligned}
		$$
		
		\noindent{\bf{(\romannumeral2)}} If $\eta=0$, then
		$$
		\begin{aligned}
			\bm{N}_0={\rm{Diag}}(\bm{w})-\frac{1}{\eta_1}({\rm{sgn}}(\bm{a}_{\mathcal{K}_2}^{\top}\bm{e}_{\mathcal{K}_2})\|\bm{a}_{\mathcal{K}_2}\|\bm{a}_{\mathcal{K}_2}^n+\sqrt{|\mathcal{K}_2|}\bm{e}_{\mathcal{K}_2}^n)({\rm{sgn}}(\bm{a}_{\mathcal{K}_2}^{\top}\bm{e}_{\mathcal{K}_2})\|\bm{a}_{\mathcal{K}_2}\|\bm{a}_{\mathcal{K}_2}^n+\sqrt{|\mathcal{K}_2|}\bm{e}_{\mathcal{K}_2}^n)^{\top},
		\end{aligned}
		$$
		where $\eta_1:=(\|\bm{a}_{\mathcal{K}_2}\|^2+|\mathcal{K}_2|)^2$.
	\end{theorem}
	\begin{proof}
		
		{\bf{\uppercase\expandafter{\romannumeral1}.}} If $\bm{a}^{\top}\Pi_{\mathcal{C}}(\bm{y})\neq b$, the matrix $\bm{B}_{\mathcal{I}(\bm{y})}$ given as in \eqref{BY} has the form: $\bm{B}_{\mathcal{I}(\bm{y})}=-\bm{I}_{\mathcal{K}_1}$.
		After calculation, we obtain
		$$
		\begin{bmatrix}
			\bm{B}_{\mathcal{I}(\bm{y})}\\
			\bm{e}^{\top}
		\end{bmatrix}[\bm{B}_{\mathcal{I}(\bm{y})}^{\top}\ \bm{e}]=\begin{bmatrix}
			\bm{I}_{|\mathcal{K}_1|}&  -\bm{e}_{\mathcal{K}_1}\\
			-\bm{e}_{\mathcal{K}_1}^{\top}&  n
		\end{bmatrix},
		$$
		which is clearly  a nonsingular matrix due  to $|{\mathcal{K}_2}|\neq 0$. Thus, by elementary row transformation, we have
		\begin{align*}
			\left(\begin{bmatrix}
				\bm{B}_{\mathcal{I}(\bm{y})}\\
				\bm{e}^{\top}
			\end{bmatrix}[\bm{B}_{\mathcal{I}(\bm{y})}^{\top}\ \bm{e}]\right)^{-1}&=\begin{bmatrix}
				\bm{I}_{|\mathcal{K}_1|}+\frac{1}{|\mathcal{K}_2|}\bm{e}_{\mathcal{K}_1}(\bm{e}_{\mathcal{K}_1})^{\top}&  \frac{1}{|\mathcal{K}_2|}\bm{e}_{\mathcal{K}_1}\\
				\frac{1}{|\mathcal{K}_2|}\bm{e}_{\mathcal{K}_1}^{\top}&  \frac{1}{|\mathcal{K}_2|}
			\end{bmatrix}\\
			&=
			\begin{bmatrix}
				\bm{I}_{|\mathcal{K}_1|}& 0\\
				0& 0
			\end{bmatrix}+\frac{1}{|\mathcal{K}_2|}\begin{bmatrix}
				\bm{e}_{\mathcal{K}_1}\\
				1
			\end{bmatrix}\begin{bmatrix}
				\bm{e}_{\mathcal{K}_1}^{\top}&1
			\end{bmatrix}.
		\end{align*}
		Then
		$$
		[\bm{B}_{\mathcal{I}(\bm{y})}^{\top}\ \bm{e}]\left(
		\begin{bmatrix}
			\bm{B}_{\mathcal{I}(\bm{y})}\\
			\bm{e}^{\top}
		\end{bmatrix}[\bm{B}_{\mathcal{I}(\bm{y})}^{\top}\ \bm{e}]\right)^{-1}\begin{bmatrix}
			\bm{B}_{\mathcal{I}(\bm{y})}\\
			\bm{e}^{\top}
		\end{bmatrix}=\bm{I}_{\mathcal{K}_1}^{\top}\bm{I}_{\mathcal{K}_1}+\frac{1}{|\mathcal{K}_2|}\bm{e}_{\mathcal{K}_2}^n(\bm{e}_{\mathcal{K}_2}^n)^{\top}.
		$$
		Therefore, invoking \eqref{BY}, we get 
		$$
		\bm{N}_0=\bm{I}_n- [\bm{B}_{\mathcal{I}(\bm{y})}^{\top}\ \bm{e}]\left(
		\begin{bmatrix}
			\bm{B}_{\mathcal{I}(\bm{y})}\\
			\bm{e}^{\top}
		\end{bmatrix}[\bm{B}_{\mathcal{I}(\bm{y})}^{\top}\ \bm{e}]\right)^{\dagger}\begin{bmatrix}
			\bm{B}_{\mathcal{I}(\bm{y})}\\
			\bm{e}^{\top}
		\end{bmatrix}= {\rm{Diag}}(\bm{w})-\frac{1}{|\mathcal{K}_2|}\bm{e}_{\mathcal{K}_2}^n(\bm{e}_{\mathcal{K}_2}^n)^{\top}.
		$$
		{\bf{\uppercase\expandafter{\romannumeral2}.}} If $\bm{a}^{\top}\Pi_{\mathcal{C}}(\bm{y})= b$, we have the explicit form for $\bm{B}_{\mathcal{I}(\bm{y})}$ as follows:
		$$
		\bm{B}_{\mathcal{I}(\bm{y})}=\begin{bmatrix}
			-\bm{I}_{\mathcal{K}_1}\\
			\bm{a}^{\top}
		\end{bmatrix}.
		$$
		Thus,
		$$
		\begin{bmatrix}
			\bm{B}_{\mathcal{I}(\bm{y})}\\
			\bm{e}^{\top}
		\end{bmatrix}[\bm{B}_{\mathcal{I}(\bm{y})}^{\top}\ \bm{e}]=\begin{bmatrix}
			-\bm{I}_{\mathcal{K}_1}\\
			\bm{a}^{\top}\\
			\bm{e}^{\top}
		\end{bmatrix}\begin{bmatrix}
			-\bm{I}_{\mathcal{K}_1}^{\top} &\bm{a}&\bm{e}
		\end{bmatrix}.
		$$	
		Denote $ \tilde{\bm{H}}:=\begin{bmatrix}
			\bm{a}^{\top}\\
			\bm{e}^{\top}
		\end{bmatrix}.$ 
		Then, together with Proposition \ref{pro10} and its proof procedure in \cite[Proposition 2]{L.S.T.2020}, one obtains
		\begin{equation}\label{N0}
			\bm{N}_0={\rm{Diag}}(\bm{w})-{\rm{Diag}}(\bm{w})\tilde{\bm{H}}^{\top}(\tilde{\bm{H}}{\rm{Diag}}(\bm{w})\tilde{\bm{H}}^{\top})^{\dagger}\tilde{\bm{H}}{\rm{Diag}}(\bm{w}).
		\end{equation}
		After a simple manipulation, we derive that
		$$
		\tilde{\bm{H}}{\rm{Diag}}(\bm{w})\tilde{\bm{H}}^{\top}=\begin{bmatrix}
			\|\bm{a}_{\mathcal{K}_2}\|^2 & \bm{a}_{\mathcal{K}_2}^{\top}\bm{e}_{\mathcal{K}_2}\\
			\bm{a}_{\mathcal{K}_2}^{\top}\bm{e}_{\mathcal{K}_2} &|\mathcal{K}_2|
		\end{bmatrix}.
		$$
		
		\noindent{\bf{(\romannumeral1)}} If $\eta:=\|\bm{a}_{\mathcal{K}_2}\|^2|\mathcal{K}_2|-(\bm{a}_{\mathcal{K}_2}^{\top}\bm{e}_{\mathcal{K}_2})^2\neq 0$, then $\tilde{\bm{H}}{\rm{Diag}}(\bm{w})\tilde{\bm{H}}^{\top}$ is nonsingular. Using the  elementary row transformation, we have
		$$
		(\tilde{\bm{H}}{\rm{Diag}}(\bm{w})\tilde{\bm{H}}^{\top})^{-1}=\frac{1}{\eta}\begin{bmatrix}
			|\mathcal{K}_2| & -\bm{a}_{\mathcal{K}_2}^{\top}\bm{e}_{\mathcal{K}_2}\\
			-\bm{a}_{\mathcal{K}_2}^{\top}\bm{e}_{\mathcal{K}_2} &\|\bm{a}_{\mathcal{K}_2}\|^2
		\end{bmatrix},
		$$
		which implies 
		$$
		\begin{aligned}
			&{\rm{Diag}}(\bm{w})\tilde{\bm{H}}^{\top}(\tilde{\bm{H}}{\rm{Diag}}(\bm{w})\tilde{\bm{H}}^{\top})^{-1}\tilde{\bm{H}}{\rm{Diag}}(\bm{w})\\
			&=\frac{1}{\eta}\left(|\mathcal{K}_2|\bm{a}_{\mathcal{K}_2}^n(\bm{a}_{\mathcal{K}_2}^n)^{\top}-(\bm{a}_{\mathcal{K}_2}^{\top}\bm{e}_{\mathcal{K}_2})\bm{e}_{\mathcal{K}_2}^n(\bm{a}_{\mathcal{K}_2}^n)^{\top}-(\bm{a}_{\mathcal{K}_2}^{\top}\bm{e}_{\mathcal{K}_2})\bm{a}_{\mathcal{K}_2}^n(\bm{e}_{\mathcal{K}_2}^n)^{\top}+\|\bm{a}_{\mathcal{K}_2}\|^2\bm{e}_{\mathcal{K}_2}^n(\bm{e}_{\mathcal{K}_2}^n)^{\top}\right)\\
			&=\frac{1}{\eta}(\sqrt{|\mathcal{K}_2|}\bm{a}_{\mathcal{K}_2}^n-\|\bm{a}_{\mathcal{K}_2}\|\bm{e}_{\mathcal{K}_2}^n)(\sqrt{|\mathcal{K}_2|}\bm{a}_{\mathcal{K}_2}^n-\|\bm{a}_{\mathcal{K}_2}\|\bm{e}_{\mathcal{K}_2}^n)^{\top}\\
			&\ +\frac{\sqrt{|\mathcal{K}_2|}\|\bm{a}_{\mathcal{K}_2}\|-\bm{a}_{\mathcal{K}_2}^{\top}\bm{e}_{\mathcal{K}_2}}{\eta}(\bm{a}_{\mathcal{K}_2}^n(\bm{e}_{\mathcal{K}_2}^n)^{\top}+\bm{e}_{\mathcal{K}_2}^n(\bm{a}_{\mathcal{K}_2}^n)^{\top}).
		\end{aligned}
		$$
		Combining this with \eqref{N0} yields
		$$
		\begin{aligned}
			\bm{N}_0=&{\rm{Diag}}(\bm{w})-\frac{1}{\eta}(\sqrt{|\mathcal{K}_2|}\bm{a}_{\mathcal{K}_2}^n-\|\bm{a}_{\mathcal{K}_2}\|\bm{e}_{\mathcal{K}_2}^n)(\sqrt{|\mathcal{K}_2|}\bm{a}_{\mathcal{K}_2}^n-\|\bm{a}_{\mathcal{K}_2}\|\bm{e}_{\mathcal{K}_2}^n)^{\top}\\
			&-\frac{\sqrt{|\mathcal{K}_2|}\|\bm{a}_{\mathcal{K}_2}\|-\bm{a}_{\mathcal{K}_2}^{\top}\bm{e}_{\mathcal{K}_2}}{\eta}(\bm{a}_{\mathcal{K}_2}^n(\bm{e}_{\mathcal{K}_2}^n)^{\top}+\bm{e}_{\mathcal{K}_2}^n(\bm{a}_{\mathcal{K}_2}^n)^{\top}).
		\end{aligned}
		$$
		\noindent{\bf{(\romannumeral2)}} If $\eta:=\|\bm{a}_{\mathcal{K}_2}\|^2|\mathcal{K}_2|-(\bm{a}_{\mathcal{K}_2}^{\top}\bm{e}_{\mathcal{K}_2})^2= 0$, then $\tilde{\bm{H}}{\rm{Diag}}(\bm{w})\tilde{\bm{H}}^{\top}$ is singular. Next, we divide our discussions into the following two cases:
		
		\noindent{\bf{Case 1:}} If $\bm{a}_{\mathcal{K}_2}\neq {\bf{0}}$, 
		$\tilde{\bm{H}}{\rm{Diag}}(\bm{w})\tilde{\bm{H}}^{\top}$ admits the following full rank factorization (cf. \cite{G.T.N.1960}):
		$$
		\tilde{\bm{H}}{\rm{Diag}}(\bm{w})\tilde{\bm{H}}^{\top}=\bm{FG}\ {\rm{with}}\ \bm{F}= \begin{bmatrix}
			\|\bm{a}_{\mathcal{K}_2}\|^2\\
			\bm{a}_{\mathcal{K}_2}^{\top}\bm{e}_{\mathcal{K}_2} 
		\end{bmatrix},\ {\rm{and}}\
		\bm{G} = \begin{bmatrix}
			1&	\displaystyle{\frac{\bm{a}_{\mathcal{K}_2}^{\top}\bm{e}_{\mathcal{K}_2}}{\|\bm{a}_{\mathcal{K}_2}\|^2}}
		\end{bmatrix}.
		$$
		Then, we compute the Moore-Penrose inverse of $\tilde{\bm{H}}{\rm{Diag}}(\bm{w})\tilde{\bm{H}}^{\top}$ by
		$$
		\begin{aligned}
			(\tilde{\bm{H}}{\rm{Diag}}(\bm{w})\tilde{\bm{H}}^{\top})^{\dagger}&=\bm{G}^{\top}(\bm{GG}^{\top})^{-1}(\bm{F}^{\top}\bm{F})^{-1}\bm{F}^{\top}\\
			&=\frac{1}{(|\mathcal{K}_2|+\|\bm{a}_{\mathcal{K}_2}\|^2)^2}\begin{bmatrix}
				\|\bm{a}_{\mathcal{K}_2}\|^2 & \bm{a}_{\mathcal{K}_2}^{\top}\bm{e}_{\mathcal{K}_2}\\
				\bm{a}_{\mathcal{K}_2}^{\top}\bm{e}_{\mathcal{K}_2} & |\mathcal{K}_2|
			\end{bmatrix}.
		\end{aligned}
		$$
		Hence, one can obtains that
		$$
		\begin{aligned}
			&{\rm{Diag}}(\bm{w})\tilde{\bm{H}}^{\top}(\tilde{\bm{H}}{\rm{Diag}}(\bm{w})\tilde{\bm{H}}^{\top})^{\dagger}\tilde{\bm{H}}{\rm{Diag}}(\bm{w})
			\\
			&=\frac{1}{\eta_1}\left(\|\bm{a}_{\mathcal{K}_2}\|^2\bm{a}_{\mathcal{K}_2}^n(\bm{a}_{\mathcal{K}_2}^n)^{\top}+(\bm{a}_{\mathcal{K}_2}^{\top}\bm{e}_{\mathcal{K}_2})\bm{e}_{\mathcal{K}_2}^n(\bm{a}_{\mathcal{K}_2}^n)^{\top}+(\bm{a}_{\mathcal{K}_2}^{\top}\bm{e}_{\mathcal{K}_2})\bm{a}_{\mathcal{K}_2}^n(\bm{e}_{\mathcal{K}_2}^n)^{\top}+|\mathcal{K}_2|\bm{e}_{\mathcal{K}_2}^n(\bm{e}_{\mathcal{K}_2}^n)^{\top}\right)\\
			&=\frac{1}{\eta_1}({\rm{sgn}}(\bm{a}_{\mathcal{K}_2}^{\top}\bm{e}_{\mathcal{K}_2})\|\bm{a}_{\mathcal{K}_2}\|\bm{a}_{\mathcal{K}_2}^n+\sqrt{|\mathcal{K}_2|}\bm{e}_{\mathcal{K}_2}^n)({\rm{sgn}}(\bm{a}_{\mathcal{K}_2}^{\top}\bm{e}_{\mathcal{K}_2})\|\bm{a}_{\mathcal{K}_2}\|\bm{a}_{\mathcal{K}_2}^n+\sqrt{|\mathcal{K}_2|}\bm{e}_{\mathcal{K}_2}^n)^{\top},
		\end{aligned}
		$$
		where $\eta_1=(\|\bm{a}_{\mathcal{K}_2}+|\mathcal{K}_2|\|^2)^2$ and the second equality holds due to 
		$$
		{\rm{sgn}}(\bm{a}_{\mathcal{K}_2}^{\top}\bm{e}_{\mathcal{K}_2})\|\bm{a}_{\mathcal{K}_2}\| |\mathcal{K}_2|=({\rm{sgn}}(\bm{a}_{\mathcal{K}_2}^{\top}\bm{e}_{\mathcal{K}_2}))^2(\bm{a}_{\mathcal{K}_2}^{\top}\bm{e}_{\mathcal{K}_2})=\bm{a}_{\mathcal{K}_2}^{\top}\bm{e}_{\mathcal{K}_2}.
		$$
		Therefore, it holds that
		$$
		\begin{aligned}
			\bm{N}_0={\rm{Diag}}(\bm{w})-\frac{1}{\eta_1}({\rm{sgn}}(\bm{a}_{\mathcal{K}_2}^{\top}\bm{e}_{\mathcal{K}_2})\|\bm{a}_{\mathcal{K}_2}\|\bm{a}_{\mathcal{K}_2}^n+\sqrt{|\mathcal{K}_2|}\bm{e}_{\mathcal{K}_2}^n)({\rm{sgn}}(\bm{a}_{\mathcal{K}_2}^{\top}\bm{e}_{\mathcal{K}_2})\|\bm{a}_{\mathcal{K}_2}\|\bm{a}_{\mathcal{K}_2}^n+\sqrt{|\mathcal{K}_2|}\bm{e}_{\mathcal{K}_2}^n)^{\top}.
		\end{aligned}
		$$
		
		\noindent{\bf{Case 2:}} If $\bm{a}_{\mathcal{K}_2}= {\bf{0}}$, $\tilde{\bm{H}}{\rm{Diag}}(\bm{w})\tilde{\bm{H}}^{\top}$ admits the following full rank factorization:
		$$
		\tilde{\bm{H}}{\rm{Diag}}(\bm{w})\tilde{\bm{H}}^{\top}=\tilde{\bm{F}}\tilde{\bm{G}}\ {\rm{with}}\ \tilde{\bm{F}}= \begin{bmatrix}
			0\\
			\sqrt{|\mathcal{K}_2|}
		\end{bmatrix},\ {\rm{and}}\
		\tilde{\bm{G}} = \begin{bmatrix}
			0&	\sqrt{|\mathcal{K}_2|}
		\end{bmatrix}.
		$$
		Then the Moore-Penrose inverse of $\tilde{\bm{H}}{\rm{Diag}}(\bm{w})\tilde{\bm{H}}^{\top}$ is given by
		$$
		\begin{aligned}
			(\tilde{\bm{H}}{\rm{Diag}}(\bm{w})\tilde{\bm{H}}^{\top})^{\dagger}=\tilde{\bm{F}}^{\top}(\tilde{\bm{G}}\tilde{\bm{G}}^{\top})^{-1}(\tilde{\bm{F}}^{\top}\tilde{\bm{F}})^{-1}\tilde{\bm{F}}^{\top}
			=\frac{1}{|\mathcal{K}_2|^2}\begin{bmatrix}
				0\\
				\sqrt{|\mathcal{K}_2|}
			\end{bmatrix} \begin{bmatrix}
				0&	\sqrt{|\mathcal{K}_2|}
			\end{bmatrix}.
		\end{aligned}
		$$
		It is not difficult to derive that
		$$
		\bm{N}_0={\rm{Diag}}(\bm{w})-\frac{1}{|\mathcal{K}_2|^2}\bm{e}_{\mathcal{K}_2}^m(\bm{e}_{\mathcal{K}_2}^m)^{\top}.
		$$
		
		As a result, combining {\bf{Case 1}} with {\bf{Case 2}}, we  know that if $\eta=0$, then 
		$$
		\begin{aligned}
			\bm{N}_0={\rm{Diag}}(\bm{w})-\frac{1}{\eta_1}({\rm{sgn}}(\bm{a}_{\mathcal{K}_2}^{\top}\bm{e}_{\mathcal{K}_2})\|\bm{a}_{\mathcal{K}_2}\|\bm{a}_{\mathcal{K}_2}^n+\sqrt{|\mathcal{K}_2|}\bm{e}_{\mathcal{K}_2}^n)({\rm{sgn}}(\bm{a}_{\mathcal{K}_2}^{\top}\bm{e}_{\mathcal{K}_2})\|\bm{a}_{\mathcal{K}_2}\|\bm{a}_{\mathcal{K}_2}^n+\sqrt{|\mathcal{K}_2|}\bm{e}_{\mathcal{K}_2}^n)^{\top},
		\end{aligned}
		$$
		where $\eta_1=(\|\bm{a}_{\mathcal{K}_2}\|^2+|\mathcal{K}_2|)^2$.
		
		With the above arguments, we complete the proof.	
	\end{proof}
	
	\section{Conlusions}\label{sect:7}
	In this paper, we  develop two efficient algorithms for finding the projection onto the intersection of  simplex and singly linear constraint. The first algorithm, referred to as LRSA, is based on the Lagrangian duality approach and the secant method. The second algorithm is an algorithm based on the semismooth Newton method, called SSN, where semismooth Newton method is developed to solve  the nonsmooth equation. Numerical results show the superior performance of the Algorithm LRSA compared to the  Algorithm SSN  and the state-of-the-art solver called Gurobi. In addition, we derive the generalized HS-Jacobian of the studied projection. 

\section{Acknowledgments}

The work of Yong-Jin Liu was in part supported by the National Natural Science Foundation of China (Grant No. 12271097),  the Key Program of National Science Foundation of Fujian Province of China (Grant No. 2023J02007), and the Fujian Alliance of Mathematics (Grant No. 2023SXLMMS01).

\section*{Declarations}

\begin{itemize}
	\item Conflict of interest: The authors declare that they have no conflict of interest.
	\item Code availability:  Code for data  analysis is available at \url{https://lcondat.github.io/download/condat_simplexproj.c} 
	\item Data availability: The data that support the findings of this study are openly available in Ken French' s website \url{https://mba.tuck.dartmouth.edu/pages/faculty/ken.french/data_library.html}
	
\end{itemize}

\end{document}